\documentclass[11pt]{amsart}

\usepackage{amsaddr}
\usepackage{cases}
\usepackage{amsfonts}
\usepackage{mathtools}
\usepackage{amscd}
\usepackage{amssymb}
\usepackage{amsmath, amsthm, amscd, amsfonts, amssymb, graphicx, color}
\usepackage{microtype}

\usepackage[UKenglish]{babel}
\usepackage[UKenglish]{datetime}

\usepackage{chngcntr}
\usepackage{cite}
\usepackage{latexsym}
\usepackage{bbm} 
\usepackage{bbold}

\usepackage[active]{srcltx}
\usepackage{tikz}
\usepackage{url}
\usepackage{verbatim} 

\usepackage{hyperref}

\usepackage{enumitem}
\usepackage{thmtools} 
\usepackage{cleveref}

\usepackage{datetime}

\usepackage{ifthen}

\makeatletter
\newcommand{\tfs}[1]
{
	\ifthenelse{\equal{\f@shape}{n}}{\ensuremath{\mathrm{#1}}}
	{\ifthenelse{\equal{\f@shape}{sc}}{\ensuremath{\mathrm{#1}}}
		{\ifthenelse{\equal{\f@shape}{it}}{\ensuremath{\mathit{#1}}}
			{\ifthenelse{\equal{\f@shape}{sl}}{\ensuremath{\mathit{#1}}}{}	
			}
		}
	}
}
\makeatother

\makeatletter
\newcommand{\btfs}[1]
{
	\ifthenelse{\equal{\f@shape}{n}}{\ensuremath{\mathrm{#1}}}
	{\ifthenelse{\equal{\f@shape}{sc}}{\ensuremath{\mathrm{#1}}}
		{\ifthenelse{\equal{\f@shape}{it}}{\ensuremath{\mathit{#1}}}
			{\ifthenelse{\equal{\f@shape}{sl}}{\ensuremath{\mathit{#1}}}{}	
			}
		}
	}
}
\makeatother

\usepackage[charter]{mathdesign}
\usepackage[mathscr]{euscript}
\newcommand{\bbfont}{\mathbbm}

\newcommand{\NN}{{\bbfont N}}
\newcommand{\RR}{{\bbfont R}}

\newcommand{\ulp}{{\textup{(}}}
\newcommand{\urp}{{\textup{)}}}
\newcommand{\uppars}[1]{\ulp #1\urp}

\newcommand{\abs}[1]{{\lvert #1 \rvert}}
\newcommand{\norm}[1]{{\lVert #1 \rVert}}
\newcommand{\rnorm}[1]{\Vert #1 \rVert_{\tfs r}}

\newcommand{\braces}[1]{{\{ #1\}}}

\newcommand{\lrabs}[1]{{\left\lvert #1 \right\rvert}}

\newcommand{\lrbraces}[1]{{\left\{ #1\right\}}}

\newcommand{\set}[1]{\braces{\,#1\,}}
\newcommand{\lrset}[1]{\lrbraces{\,#1\,}}

\newcommand{\di}[1]{\,{\tfs d} #1}

\newcommand{\bounded}{{\tfs B}}
\newcommand{\linear}{{\tfs L}}
\newcommand{\regular}{\linear_{\tfs r}}
\newcommand{\ob}{\linear_{\tfs b}}
\newcommand{\nob}{\linear_{\tfs{nob}}}
\newcommand{\soc}{\linear_{\sigma\tfs{oc}}}

\newcommand{\unitball}{{\tfs B}_1}

\theoremstyle{plain}

\newtheorem{theorem0}{Theorem}[section]
\newtheorem{theorem}[theorem0]{Theorem}
\newtheorem{proposition}[theorem0]{Proposition}
\newtheorem{lemma}[theorem0]{Lemma}
\newtheorem{corollary}[theorem0]{Corollary}

\newtheorem*{theorem*}{Theorem}
\newtheorem*{proposition*}{Proposition}
\newtheorem*{lemma*}{Lemma}
\newtheorem*{corollary*}{Corollary}
\newtheorem*{conjecture*}{Conjecture}

\theoremstyle{definition}

\newtheorem{definition}[theorem0]{Definition}

\newtheorem{remark}[theorem0]{Remark}

\newtheorem*{definition*}{Definition}
\newtheorem*{example*}{Example}
\newtheorem*{remark*}{Remark}

\setlist[enumerate,1]{label=\textup{(\arabic*)},ref=\arabic*}
\setlist[enumerate,2]{label=\textup{(\alph*)},ref=\arabic{enumi}.\alph*}
\setlist[enumerate,3]{label=\textup{(\roman*)},ref=\arabic{enumi}.\alph{enumii}.\roman*}
\setlist[enumerate,4]{label=\textup{(\Alph*)},ref=\arabic{enumi}.\alph{enumii}.\roman{enumiii}.\Alph*}

\crefname{theorem}{Theorem}{Theorems}
\crefname{proposition}{Proposition}{Propositions}
\crefname{lemma}{Lemma}{Lemmas}
\crefname{corollary}{Corollary}{Corollaries}
\crefname{conjecture}{Conjecture}{Conjectures}
\crefname{definition}{Definition}{Definitions}
\crefname{example}{Example}{Examples}
\crefname{remark}{Remark}{Remarks}

\crefname{section}{Section}{Sections}
\crefname{subsection}{Section}{Sections}
\crefname{subsubsection}{Section}{Sections}

\crefname{equation}{equation}{equations}
\crefname{enumi}{part}{parts}
\crefname{enumii}{part}{parts}
\crefname{enumiii}{part}{parts}
\crefname{enumiv}{part}{parts}

\newcommand{\enclosepart}[1]{(#1)}
\newcommand{\partref}[1]{\enclosepart{\ref{#1}}}

\numberwithin{equation}{section}
\allowdisplaybreaks 

\newcommand{\Dc}{De\-de\-kind com\-plete}
\newcommand{\sDc}{$\sigma$-De\-de\-kind com\-plete}

\newcommand{\setsup}{\sup}
\newcommand{\setinf}{\inf}

\newcommand{\supp}{\mathrm{supp}\,}

\newcommand{\zerofunction}{\mathbf{0} }
\newcommand{\onefunction}{\mathbf{1}}
\newcommand{\indicator}[1]{\chi_{#1}}

\newcommand{\pos}[1]{{#1^+}}

\newcommand{\largest}{\infty}

\newcommand{\seq}[1]{\{{#1}_n\}_{n=1}^{\infty}}
\newcommand{\net}[1]{\{{#1}_\lambda\}_{\lambda\in \Lambda}}

\newcommand{\Ell}{{\tfs L}}

\newcommand{\os}{E}
\newcommand{\bl}{E}
\newcommand{\vl}{E}
\newcommand{\vltwo}{F}

\newcommand{\posop}{T}
\newcommand{\posoptwo}{S}

\newcommand{\pset}{X}
\newcommand{\ts}{X}

\newcommand{\posos}{\pos{\os}}
\newcommand{\osext}{\overline{\os}}
\newcommand{\pososext}{\overline{\posos}}

\newcommand{\posR}{\pos{\RR}}

\newcommand{\posRext}{\overline{\pos{\RR}}}

\newcommand{\posmap}{T}

\newcommand{\cont}[1]{{\tfs{C}}(#1)}
\newcommand{\conto}[1]{\tfs{C}_0(#1)}
\newcommand{\contc}[1]{\tfs{C}_{\tfs{c}}(#1)}
\newcommand{\contts}{\cont{\ts}}
\newcommand{\contots}{\conto{\ts}}
\newcommand{\contcts}{\contc{\ts}}

\newcommand{\odual}[1]{{#1^{\sim}}}

\newcommand{\ocdual}[1]{{#1_{\tfs {oc}}^{\sim}}}

\newcommand{\odualos}{\odual{\os}}

\newcommand{\ocdualos}{\ocdual{\os}}

\newcommand{\obounded}{\linear_{\tfs b}}
\newcommand{\ocontinuous}{\linear_{\tfs{oc}}}
\newcommand{\socontinuous}{\linear_{\sigma\tfs{oc}}}

\newcommand{\alg}{\Omega}
\newcommand{\borel}{\mathscr B}

\newcommand{\mss}{\Delta}
\newcommand{\msstwo}{\Gamma}

\newcommand{\ms}{(\pset,\alg)}

\newcommand{\npm}{\mu}
\newcommand{\npn}{\nu}

\newcommand{\spaceofmeasuresletter}{{\tfs M}}

\newcommand{\posmeas}{{\spaceofmeasuresletter(\pset,\alg,\posos)}}
\newcommand{\posextmeas}{{\spaceofmeasuresletter(\pset,\alg,\pososext)}}
\newcommand{\posmeasts}{{\spaceofmeasuresletter(\ts,\borel,\posos)}}
\newcommand{\posextmeasts}{{\spaceofmeasuresletter(\ts,\borel,\pososext)}}

\newcommand{\posrBmeas}{{\spaceofmeasuresletter_{\tfs {rB}}(\ts,\borel,\posos)}}
\newcommand{\posextrBmeas}{{\spaceofmeasuresletter_{\tfs {rB}}(\ts,\borel,\pososext)}}
\newcommand{\smeas}{{\spaceofmeasuresletter(\pset,\alg,\vl)}}
\newcommand{\smeasts}{{\spaceofmeasuresletter(\ts,\borel,\vl)}}
\newcommand{\rBsmeas}{{\spaceofmeasuresletter_{\tfs{rB}}(\ts,\borel,\vl)}}

\newcommand{\boundedmeasfun}  {{\mathcal B}(\pset,\alg,\RR)}

\newcommand{\orderintegral}[3]{{\int_{#1}^{\tfs {o}}\! {#2}\di {#3}}}
\newcommand{\ointm}[1]{\orderintegral{\pset}{#1}{\npm}}
\newcommand{\opint}[1]{I_{#1}}

\begin{document}

	\title[Riesz representation theorems]{Riesz representation theorems for vector lattices and Banach lattices of regular  operators}

	\author{Marcel de Jeu}
	
	\address{Mathematical Institute, Leiden University, P.O.\ Box 9512, 2300 RA Leiden, The Netherlands\\
		and\\
		Department of Mathematics and Applied Mathematics, University of Pretoria, Corner of Lynnwood Road and Roper Street, Hatfield 0083, Pretoria,
		South Africa}
	
	\email{mdejeu@math.leidenuniv.nl}
	
	\thanks{This work was completed with the support of National Natural Science Foundation of China(NSFC12201439), Natural Science Foundation of Sichuan Province (2024NSFSC1339) and Erasmus+ICM programme (KA171). }
	\author{Xingni Jiang}
	\address{College of Mathematics, Sichuan University, No.\ 24, South Section, First Ring Road, Chengdu, P.R.\ China}
	\email{x.jiang@scu.edu.cn}
	
	
	
	\date{January 5, 2025}
	
	\begin{abstract}
		For a non-empty locally compact Hausdorff space $X$ and a Dedekind complete normal vector lattice $E$, we show that the vector lattice of norm to order bounded operators from ${\text C}_{\text c}(X)$ or ${\text C}_0(X)$ into $E$ is isomorphic to the vector lattice of $E$-valued regular Borel measures on $X$. When $E$ is an order continuous Banach lattice, the isomorphism is an isometric isomorphism between Banach lattices. When $X$ is compact, every regular operator from $\mathrm{C}(X)$ into $E$ is norm to order bounded. For some spaces $E$, such as KB-spaces or the regular operators on a KB-space, every regular operator from ${\mathrm C}_0(X)$ into $E$ is norm to order bounded. Additional results are obtained for the whole space of regular operators from ${\text C}_{\text c}(X)$  into an order continuous Banach lattice.\\
		\noindent As a preparation, vector lattices and Banach lattices, resp.\ cones, of measures with values in a Dedekind complete vector lattice $E$, resp.\ in the extended positive cone of $E$, are investigated, as well as vector and Banach lattices of norm to order bounded operators.\\
		\noindent When $E$ is the real numbers, our results specialise to the well-known Riesz representation theorems for the order and norm duals of ${\text C}_{\text c}(X)$ and ${\text C}_0(X)$.
	\end{abstract}
	
	\maketitle
	\section{Introduction and overview}\label{sec:introduction}
	
	\noindent Let $\ts$ be a non-empty locally compact Hausdorff space. When $\posop\colon\contcts\to\RR$ is a continuous positive functional, the Riesz representation theorem furnishes a unique finite regular Borel measure $\npm$ such that
	\begin{equation}\label{eq:riesz_representation}
		T(f)=\int_\ts f\di{\npm(x)}
	\end{equation}
	for $f\in\contcts$. Developing the theory further, one shows that the norm dual of $\contcts$ is isometrically isomorphic to the Banach lattice of real-valued regular Borel measures on $\ts$, supplied with the total variation norm. In the current paper, similar isomorphisms are given for vector lattices and Banach lattices of regular operators $\posop\colon \contcts\to\vl$ or $\posop\colon\contots\to\vl$ into vector lattices or Banach lattices $\vl$. The measures are then $\vl$-valued. As will become clear in the overview of the paper, one cannot expect that, for example, the vector lattice of \emph{all} regular operators from $\contcts$ into a general \Dc\ vector lattice $\vl$ is isomorphic to a vector lattice of $\vl$-valued measures via the integral.
This is only possible for vector lattices of norm to order bounded operators. The "actual" theorem when $\vl=\RR$ is that the Banach lattice of real-valued regular Borel measures on $\ts$ and that of the norm to order bounded operators from $\contcts$ to $\RR$ are isometric via the integral.\footnote{See \cref{res:order_continuous_banach_lattice_2}.} Because in $\RR$ the norm and order bounded subsets coincide, the latter Banach lattice is also the norm dual of $\contcts$.
	
	When $\posop\colon\contcts\to\RR$ is positive but not necessarily continuous, i.e, not necessarily norm to order bounded, there is still a regular Borel measure representing $\posop$ as in \cref{eq:riesz_representation}, but it need not be finite. The order dual of $\contcts$ can then be described in terms of two measures, possibly both infinite. We generalise this to a description of the regular operators $\posop\colon\contcts\to\vl$ when $\vl$ is an order continuous Banach lattice.

	This paper is organised as follows.
	
	In \cref{sec:preliminaries}, we include the necessary notation and definitions. We also give a concise definition of the so-called order integral of real-valued functions with respect to a measure taking values in the extended positive cone of a \sDc\ vector lattice. This integral will be used to formulate our representation theorems with. The idea of such an integral goes back to the seminal work by Wright; see \cite{wright:1969a,wright:1969b}, among others. A systematic treatment of the theory of the order integral was included in the second author's 2018 PhD thesis \cite{jiang_THESIS:2018} and appeared as a paper  \cite{de_jeu_jiang:2022a} in 2022. Meanwhile, Kusraev and Tasoev independently developed the theory of the closely related so-called Kantorovich-Wright integral, which was published as a part of their 2017 paper \cite{kusraev_tasoev:2017}. Comparing \cite{kusraev_tasoev:2017} and \cite{de_jeu_jiang:2022a}, there is a technical difference in the domains of the measures (initially $\delta$-rings of sets in the former; initially rings of sets in the latter) but, more importantly, the measures in \cite{kusraev_tasoev:2017} are finite whereas they can be infinite in \cite{de_jeu_jiang:2022a}. For our purposes, it is essential that a measure can be infinite to give representation theorems where finite measures cannot suffice.\footnote{See \cref{res:order_continuous_banach_lattice_1}.} The Kantorovich-Wright integral is used in \cite[Theorem~5.1]{kusraev_tasoev:2017} to show that a \sDc\ vector lattice with a weak order unit is isomorphic to the $\Ell^1$-space associated with a suitable measure via the ensuing integration operator. We refer to \cite{kusraev_tasoev:2017, kusraev_tasoev:2018} for further such representation theorems for vector lattices and quasi-Banach lattices. The order integral from \cite{de_jeu_jiang:2022a} is employed in \cite{de_jeu_jiang:2022b} for representation theorems for positive operators.\footnote{See \cref{res:riesz_representation_theorem_for_contcts_finite_normal_case} and \cref{res:riesz_representation_theorem_for_contcts_normed_case}.}
	These are used in \cite{de_jeu_jiang:2021c} to arrive at spectral measures for positive algebra homomorphisms, and are also basic ingredients for the current paper.\footnote{For a compact Hausdorff space $\ts$, it is shown in \cite{tamaeva_tasoev:2024} that the representing measure for a positive $\posop\colon\contts\to\vl$ is a Boolean homomorphism precisely when $\posop$ is a vector lattice homomorphism. There are similar results in \cite[Section~4]{de_jeu_jiang:2021c}.}

	The inclusion of \cref{sec:norm_to_order_bounded_operators} is motivated by the following observation. Suppose that $\vl$ is a \Dc\ vector lattice, that $\npm$ is the difference of two finite measures on a $\sigma$-algebra of subsets of a set $\pset$ taking values in $\posos$, and that $f\colon\pset\to\RR$ is a bounded measurable function. Then the triangle inequality for the order integral implies that
	\[
	\lrabs{\ointm{f}}\leq\int_\pset\abs{f}\,\di{\abs{\mu}}\leq \norm{f}\, {\abs{\npm}}(\pset).
	\]
	Thus, if $\posop$ is an operator from a vector lattice $\vltwo$ of bounded functions on $\ts$ into $\vl$, and if there is to be a finite measure $\npm$ representing $\posop$ in the sense that $\posop(f)=\ointm{f}$ for all $f\in\vltwo$, then $\posop$ must map subsets of $\vltwo$ that are bounded in the (uniform) norm into order bounded subsets of $\vl$. Therefore, if there is to be any hope for a vector lattice of regular operators from $\vltwo$ into $\vl$ to be isomorphic to a vector lattice of $\vl$-valued measures via the order integral, then it must consist of norm to order bounded operators. \cref{sec:norm_to_order_bounded_operators} is concerned with these, in the more general context of operators from a normed vector lattice into a \Dc\ vector lattice. Norm to order bounded operators are regular, but the inclusion may be strict. It is shown that, for a so-called quasi-perfect vector lattice $\vl$, every regular operator from $\contots$ for a non-empty locally compact Hausdorff space is norm to order bounded.

	\cref{sec:cones_and_vector_lattices_of_measures} is concerned with various cones, resp.\ vector lattices, of measures, with values in the extended positive cone of a \Dc\ vector lattice $\vl$, resp.\ in a \Dc\ vector lattice $\vl$. The results obtained are in the same vein as those for $\os=\RR$, but the proofs differ substantially from the ones as usually given in that case. The familiar $\varepsilon$-arguments for $\RR$ exploit the fact that $\RR$ is a vector lattice (or sometimes that it is a Banach lattice) with a one-dimensional positive cone. In the general case, one has to proceed differently and argue purely in terms of the ordering on (the extended positive cone of) a vector lattice. The necessity of proceeding cautiously and giving explicit proofs is perhaps best illustrated by the fact that the generalisation of a known result for $\RR$ is, in fact, false.\footnote{See \cref{rem:counter_example}.}
	
	In the final \cref{sec:riesz_representation_theorems}, the previous sections are combined with two basic representation theorems from \cite{de_jeu_jiang:2022b} to yield Riesz representation theorems for various vector lattices and cones of regular operators with $\contcts$ or $\contots$ as domains. It also contains a representation theorem where the domain of the regular operators consists of bounded measurable functions,  the easy proof of which is independent of the other material in this paper and of \cite{de_jeu_jiang:2022b}. Via the isomorphism in this representation theorem one can see the familiar formulas for the supremum and infimum of two finite measures as Riesz-Kantorovich formulas for operators.

	\section{Preliminaries }\label{sec:preliminaries}
	
	\noindent In this section, we collect the necessary notation, definitions, conventions, and preliminary results.
	
	All vector spaces are over the real numbers. An operator between two vector spaces is linear. Vector lattices are supposed to be Archimedean. For a vector lattice $\vl$, we let $\posos$ denote its positive cone and $\odualos$ its order dual. When $\vl$ and $\vltwo$ are vector lattices, we write $\obounded(\vl,\vltwo)$, $\regular(\vl,\vltwo)$, $\ocontinuous(\vl,\vltwo)$, and $\socontinuous(\vl,\vltwo)$ for the operators from $\vl$ into $\vltwo$ that are, respectively, order bounded; regular; order bounded and order continuous; and order bounded and $\sigma$-order continuous. In general, $\regular(\vl,\vltwo)\subseteq\obounded(\vl,\vltwo)$. When $\vltwo$ is \Dc, $\regular(\vl,\vltwo)=\obounded(\vl,\vltwo)$ and this space is a \Dc\ vector lattice; see \cite[Theorem~1.18]{aliprantis_burkinshaw_POSITIVE_OPERATORS_SPRINGER_REPRINT:2006}.
	
	When $\vl$ and $\vltwo$ are normed spaces, we let $\bounded(\vl,\vltwo)$ denote the bounded operators from $\vl$ into $\vltwo$. When $\vl$ is a normed space, we write $\vl^\ast$ for its norm dual, $\unitball(\vl)$ for its unit ball, and $\pos{\unitball(\vl)}$ for the positive elements of its unit ball when $\vl$ is a normed vector lattice.
	
	Suppose that $\vl$ and $\vltwo$ are Banach lattices. Then $\regular(\vl,\vltwo)\subseteq \bounded(\vl,\vltwo)$; see \cite[Theorem~4.3]{aliprantis_burkinshaw_POSITIVE_OPERATORS_SPRINGER_REPRINT:2006}. If $\vltwo$ is \Dc, so that $\regular(\vl,\vltwo)$ is a vector lattice, the \emph{regular norm} on $\regular(\vl,\vltwo)$ is defined by setting $\rnorm{T}\coloneqq \norm{\abs{T}}$ for $\posmap\in\regular(\vl,\vltwo)$. Supplied with this norm,  $\regular(\vl,\vltwo)$ is a \Dc\ Banach lattice; see \cite[Theorem~4.74]{aliprantis_burkinshaw_POSITIVE_OPERATORS_SPRINGER_REPRINT:2006}.
	
	When $\ts$ is a locally compact Hausdorff space, we let $\borel$ denote its Borel $\sigma$-algebra, i.e., the $\sigma$-algebra that is generated by the open subsets of $\ts$. We write $\contcts$ for the vector lattice of real-valued continuous functions on $\ts$ with compact support, and $\contots$ for the Banach lattice of real-valued continuous functions in $\ts$ that vanish at infinity.
	
	Let $S$ be a non-empty set. We let $\zerofunction$ resp.\ $\onefunction$ denote the real-valued function on $S$ that is 0 resp.\ 1 everywhere. For a subset $S^\prime$ of $S$, we let $\chi_{S^\prime}\colon S\to\{0,1\}$ denote the characteristic function of $S^\prime$ on $S$.
	
	For a vector lattice $\vl$, we form a disjoint union $\osext\coloneqq\os\cup\{\largest\}$, and call the elements of $\osext$ that are in $\os$ \emph{finite}. We extend the partial ordering from $\os$ to $\osext$ by declaring that $x\leq\infty$ for all $x\in\osext$. The addition is extended from $\os$ to $\osext$ by setting $x+\infty\coloneqq\infty$ and $\infty+x\coloneqq\infty$ for $x\in\osext$. We let $\pososext\coloneqq\posos\cup\{\infty\}$ and extend the action of $\posR$ on $\posos$ to an action on $\pososext$ by setting $0\cdot \infty\coloneqq 0$ and $r\cdot \infty\coloneqq\infty$ when $r>0$.
	
	As in \cite{de_jeu_jiang:2022a}, we collect a few basic technical facts that will be used repeatedly in the sequel. The proofs are elementary.
	\begin{lemma} \label{res:operations_in_extended_space}
		Let $\os$ be a vector lattice, and let $S\subseteq \osext$ be non-empty.
		\begin{enumerate}
			\item If $S\subseteq E$ and  $\setsup S$ exists in $\os$, then this supremum is also the supremum of $S$ in $\osext$; likewise for the infimum.\label{part:operations_in_extended_space_1}
			\item If $\setsup S$ exists in $\osext$ and is finite, then $S$ consists of finite elements, and the supremum of $S$  exists in $\os$ and equals the supremum of $S$ in $\osext$.\label{part:operations_in_extended_space_2}
			\item If $S\neq\{\largest\}$, then $\setinf S$ exists in $\osext$ if and only if $\setinf\, (S\cap E)$ exists in $\os$. If this is the case, then these infima are equal.\label{part:operations_in_extended_space_3}
			\item $\setsup S=\infty$ in $\osext$ if and only if $S$ is not bounded from above by a finite element.\label{part:operations_in_extended_space_4}
			\item If $\setsup S$ exists in $\osext$, then, for all $x\in \osext$,  $\setsup\, (x+S)$ exists in $\osext$ and equals $x+\setsup S$; likewise for the infimum.\label{part:operations_in_extended_space_5}
			\item If $\setsup\, (x+S)$ exists in $\osext$ for some $x\in\os$, then $\setsup S$ exists in $\osext$, and $\setsup\, (x+S)=x+\setsup S$; likewise for the infimum.\label{part:operations_in_extended_space_6}
			\item If $\setsup S$ exists in $\osext$, then, for all $r\in\posR$, $\setsup rS$ exists in $\osext$ and equals $r\setsup S$; likewise for the infimum.\label{part:operations_in_extended_space_7}
		\end{enumerate}
	\end{lemma}
	
	\begin{lemma}\label{res:binary_operations_in_extended_space}
		Let $\os$ be a vector lattice.
		\begin{enumerate}
			\item If $A$ and $B$ are two non-empty subsets of $\osext$ such that $\setsup A$ and $\setsup B$ exist in $\osext$, then $\setsup\, (A+B)$ exists in $\osext$ and equals $\setsup A+\setsup B$; likewise for the infima. \label{part:binary_operations_in_extended_space_1}
			\item If $\net{a}$ and $\net{b}$ are nets in  $\osext$ such that $a_\lambda\uparrow a$ and $b_\lambda\uparrow b$ in $\osext$, then $(a_\lambda+b_\lambda)\uparrow (a+b)$ in $\osext$; likewise for decreasing nets.\label{part:binary_operations_in_extended_space_2}
		\end{enumerate}
	\end{lemma}
	
	\begin{lemma}\label{res:completeness_properties_of_extended_space}
		Let $\os$ be a vector lattice.
		\begin{enumerate}
			\item If $\os$ is \sDc, then every countable non-empty subset of $\osext$ has a supremum in $\osext$. If the subset is bounded from above by a finite element, then the supremum in $\osext$ equals the supremum in $\os$.\label{part:completeness_properties_of_extended_space_1}
			\item If $\os$ is \Dc, then every non-empty subset of $\osext$ has a supremum in $\osext$. If the subset is bounded from above by a finite element, then the supremum in $\osext$ equals the supremum in $\os$. \label{part:completeness_properties_of_extended_space_3}
		\end{enumerate}
	\end{lemma}

	In the remainder of this section, we briefly summarise the relevant facts about $\pososext$-valued measures and the integrals that they define on measurable functions with values in the extended positive real numbers and on measurable real-valued functions. The pertinent theory is a generalisation of the one for the Lebesgue integral. A \emph{measurable space} is a pair $\ms$, where $\pset$ is a set and $\alg$ is a $\sigma$-algebra of subsets of $\pset$.
	
	\begin{definition}\label{def:positive_pososext_valued_measure}
		Let $\ms$ be a measurable space, and let $\os$ be a \sDc\  vector lattice. An \emph{$\pososext$-valued measure} is a map $\npm:\alg\rightarrow \pososext$ such that:
		\begin{enumerate}
			\item $\npm(\emptyset)=0$;\label{part:pososext_valued_measure_1}
			\item ($\sigma$-additivity) if $\seq{\mss}$ is a pairwise disjoint sequence in $\alg$, then
			\begin{equation}\label{eq:sigma_additivity}
				\npm\left(\bigcup_{n=1}^\infty\mss_n\right)=\setsup_{N\geq 1}\sum_{n=1}^N\npm(\mss_n)
			\end{equation}
			in $\pososext$.\label{part:pososext_valued_measure_2}
		\end{enumerate}
	\end{definition}
	
	It follows from \cref{res:completeness_properties_of_extended_space} that the supremum in \cref{eq:sigma_additivity} exists.
	
	When $\npm(\pset)\in\posos$, we say that $\npm$ is \emph{finite}, or that it is \emph{$\posos$-valued}. In this case, $\npm(\mss)\in\posos$ for all  $\mss\in\alg$. When $\npm(\pset)=\largest$, $\npm$ it is said to be \emph{infinite}. We shall write $\posextmeas$ for the $\pososext$-valued measures, and $\posmeas$ for the $\posos$-valued measures.
	
	As for the case where $\os=\RR$, there are various regularity properties to be considered for measures on the Borel $\sigma$-algebra $\borel$ of a locally compact Hausdorff space.

	\begin{definition}\label{def: regularity of measures}
		Let $\ts$ be a locally compact Hausdorff space, let $\os$ be a \Dc\ vector lattice, and let $\npm\in\posextmeasts$. Then $\npm$ is called:
		\begin{enumerate}
			\item an \emph{$\pososext$-valued Borel measure \uppars{on $\ts$}} if $\npm(K)\in \os$ for all compact subsets $K$ of $\ts$;
			\item \emph{inner regular at $\mss\in\borel$} if $\npm(\mss)=\setsup\set{\hskip -1pt\npm(K):\hskip -.5pt K \hskip -1pt \text{ is compact and}\ K\subseteq \mss\hskip -1pt}$ in $\pososext$;
			\item \emph{outer regular at $\mss\in\borel$} if $\npm(\mss)=\setinf\set{\npm(V): V\ \text{is open and}\ \mss\subseteq V}$ in $\pososext$;
			\item a \emph{$\pososext$-valued regular Borel measure \uppars{on $\ts$}} if $\npm$ is an $\pososext$-valued Borel measure on $\ts$ that is inner regular at all open subsets of $\ts$ and outer regular at all Borel sets.\footnote{We have followed the terminology as in \cite{aliprantis_burkinshaw_PRINCIPLES_OF_REAL_ANALYSIS_THIRD_EDITION:1998} for $\os=\RR$. When following \cite{folland_REAL_ANALYSIS_SECOND_EDITION:1999}, our $\pososext$-valued regular Borel measures would have been called $\pososext$-valued Radon measures.}
		\end{enumerate}
	\end{definition}
	
	We shall write $\posextrBmeas$ for the $\pososext$-valued regular Borel measures on  $\ts$, and $\posrBmeas$ for the $\posos$-valued regular Borel measures on $\ts$.
	
	We return to the general context of a measure space $\ms$, a \sDc\ vector lattice $\os$, and an $\pososext$-valued measure $\npm\colon\alg\to\pososext$.
	An \emph{elementary} function is a function $\varphi\colon\pset\to\posR$ of the form $\varphi=\sum_{i=1}^n r_i\indicator{\mss_i}$ with $r_1,\dotsc,r_n\in\RR$ and $\mss_1,\dotsc,\mss_n\in\alg$, and where the $\mss_i$ can have infinite measure. The \emph{order integral} of $\varphi$ with respect to $\npm$ is then defined as an element of $\pososext$ by setting $\ointm{\varphi}\coloneqq\sum_{i=1}^n r_i\npm(\mss_i)$; this definition is independent of the choice of a decomposition $\varphi=\sum_{i=1}^n r_i\indicator{\mss_i}$. For an arbitrary measurable $f\colon\ts\to\posRext$, one chooses a sequence $\varphi_n$ of elementary functions such that $\varphi_n(x)\uparrow f(x)$ in $\posRext$ for every $x\in\ts$, and defines  $\ointm{f}\coloneqq\sup_{n\geq 1} \ointm{\varphi_n}\in\pososext$. This is independent of the choice of the sequence. Finally, for a measurable $f\colon\pset\to\RR$ such that $\ointm{f^\pm}<\infty$, one sets $\ointm{f}\coloneqq\ointm{f^+}-\ointm{f^-}$. We shall write $\opint{\npm}(f)$ for $\ointm{f}$.
	
	The further theory for this integral is developed in some detail in \cite{de_jeu_jiang:2022a}, including results such as the monotone convergence theorem, Fatou's lemma, and the dominated convergence theorem. The triangle inequality
	\begin{equation}\label{eq:triangle_inequality}
		\lrabs{\ointm{f}}\leq \ointm{\abs{f}}
	\end{equation}
	holds for all integrable $f\colon\pset\to\RR$; see \cite[Lemma~6.7]{de_jeu_jiang:2022a}.

	\section{Norm to order bounded operators}\label{sec:norm_to_order_bounded_operators}
	
	\noindent Let $\vl$ be a normed vector lattice, and let $\vltwo$ be a vector lattice. An operator $\posop\colon\vl\to\vltwo$ is \emph{norm to order bounded} when it maps norm bounded subsets of $\vl$ into order bounded subsets of $\vltwo$. Equivalently, there should exist a $t\in\vltwo$ such that $\abs{\posop  x}\leq\norm{x}t$ for $x\in\vl$.  As motivated in \cref{sec:introduction}, such operators occur naturally in the context of operators from spaces of bounded real-valued functions into vector lattices that can be represented by measures. We let $\nob(\vl,\vltwo)$ denote the set of norm to order bounded operators from $\vl$ into $\vltwo$. It is a linear subspace of $\ob(\vl,\vltwo)$, and of $\bounded(\vl,\vltwo)$ when $\vltwo$ is also a normed vector lattice.  For a normed vector lattice $\vl$, $\nob(\vl,\RR)=\vl^\ast$.
	
	Suppose that a vector lattice $\vl$ has an order unit $u$. When $\vl$ is supplied with the order unit norm, we have $\regular(\vl,\vltwo)\subset \nob(\vl,\vltwo)$, so that $\nob(\vl,\vltwo)=\regular(\vl,\vltwo)$ when $\vltwo$ is \Dc. In \cref{res:regular_is_nob_for_quasi-perfect} we shall encounter another situation where this equality holds but where $\vl$ need not have an order unit.
	
	\begin{proposition}\label{res:nob_operators_are_ideal}
		Let $\vl$ be a normed vector lattice, and let $\vltwo$ be a \Dc\ vector lattice. Then $\nob(\vl,\vltwo)$ is an ideal in $\regular(\vl,\vltwo)$.
		
		Take $\posop\in\nob(\vl,\vltwo)$ and $t\in\posos$ such that $\abs{\posop x}\leq\norm{x}t$ for $x\in\vl$. Then:
		\begin{enumerate}
			\item $\abs{T}x\leq\norm{x}t$ for $x\in\posos$;
			\item If $\posoptwo\in\regular(\vl,\vltwo)$ and $\abs{\posoptwo}\leq\abs{\posop}$, then $\abs{Sx}\leq 2\norm{x}t$ for $x\in\vl$.
		\end{enumerate}
		Set $t_\posop\coloneqq \sup_{x\in\pos{\unitball(\vl)}}\abs{T}x$. Then $\abs{\posop x}\leq\norm{x}t_\posop$ for $x\in\vl$. If $\vltwo$ is an order continuous Banach lattice and $\pos{\unitball(\vl)}$ is upward directed, then $\rnorm{\posop}=\norm{t_\posop}$.
	\end{proposition}
	
	\begin{proof}
		Take $x\in\posos$. By the Riesz-Kantorovich formula (see \cite[Theorem~1.18]{aliprantis_burkinshaw_POSITIVE_OPERATORS_SPRINGER_REPRINT:2006}), we have  $\abs{T}x=\sup\set{\abs{\posop y}:\abs{y}\leq x}$. If $\abs{y}\leq x$, then $\abs{\posop y}\leq\norm{y}t\leq\norm{x}t$. Hence $\abs{T}x\leq\norm{x}t$ for $x\in\posos$. Suppose that $\posoptwo\in\regular(\vl,\vltwo)$ and $\abs{\posoptwo}\leq\abs{\posop}$. Then $\abs{Sx}\leq\abs{S}\abs{x}\leq\abs{T}\abs{x}\leq\abs{\posop}{x^+}+\abs{\posop}x^-\leq 2\norm{x}t$ for $x\in\vl$.
		
		For $x\in\unitball(\vl)$, we have $\abs{\posop x}\leq\abs{\posop}\abs{x}\leq\sup_{x\in\pos{\unitball(\vl)}}\abs{T}x=t_\posop$, so that $\abs{\posop x}\leq\norm{x}t_\posop$  for $x\in\vl$. If $\vltwo$ is an order continuous Banach lattice and $\pos{\unitball(\vl)}$ is upward directed, then $\rnorm{\posop}=\norm{\,\abs{\posop}\,}=\sup_{x\in\pos{\unitball(\vl)}}\norm{\,\abs{\posop}x\,}=\norm{\sup_{x\in\pos{\unitball(\vl)}}\abs{T}x}=\norm{t_T}$.
	\end{proof}
	
	Clearly, $\pos{\unitball(\vl)}$ is upward directed for every vector sublattice of an AM-space\footnote{We recall that a Banach lattice $\bl$ is said to be an \emph{AM-space} when $x\wedge y=0$ in $\bl$ implies that $\norm{x\vee y}=\norm{x}\vee\norm{y}$. A Banach lattice is an AM-space if and only if it is lattice isometric to a Banach sublattice of $\cont{K}$ for a compact Hausdorff space $K$, and  it is lattice isometric to such a $\cont{K}$-space if and only if it has a strong order unit; see \cite[Theorem~4.29]{aliprantis_burkinshaw_POSITIVE_OPERATORS_SPRINGER_REPRINT:2006}.}. The examples we have in mind in the current paper are $\contcts$ and $\contots$.
	
	\begin{proposition}\label{res:norm_to_order_bounded_closed_ideal}
		Let $\vl$ be a normed vector lattice such that $\pos{\unitball(\vl)}$ is upward directed, and let $\vltwo$ be an order continuous Banach lattice. Then $\nob(\vl,\vltwo)$ is a norm closed ideal in the Banach lattice $\regular(\vl,\vltwo)$.
	\end{proposition}	
	
	\begin{proof}
		It follows from \cref{res:nob_operators_are_ideal} that $\nob(\vl,\vltwo)$ is an ideal in $\regular(\vl,\vltwo)$; we shall show that it is norm closed. For this, suppose that  $\seq{\posop}\subseteq\nob(\vl,\vltwo)$ is such that $\sum_{n=1}^\infty\rnorm{\posop_n}<\infty$. The series $\sum_{n=1}^\infty T_n$ is norm convergent in $\regular(\vl,\vltwo)$ in the regular norm; we have to show that its sum is norm to order bounded.
		
		In the notation of  \cref{res:nob_operators_are_ideal}, we have $\sum_{n=1}^\infty\norm{t_{\posop_n}}<\infty$, so we can define $t\coloneqq\sum_{n=1}^\infty t_{\posop_n}\in\pos{\vl}$. Take $x\in\vl$. Then
		\begin{align*}
			\lrabs{\left(\sum_{n=1}^\infty T_n\right) x }&=\lrabs{\sum_{n=1}^\infty T_n x}\\
			&\leq \sum_{n=1}^\infty \abs{T_n x}\\
			&\leq \sum_{n=1}^\infty \norm{x}t_{\posop_n}\\
			&=\norm{x} t.
		\end{align*}
		Hence $\sum_{n=1}^\infty T_n$ is norm to order bounded.	
	\end{proof}

	The routine proof of the following result is omitted.
	
	\begin{lemma}\label{res:restriction_is_isomorphism}
		Let $\vl_0$ be a norm dense vector sublattice of a normed vector lattice $\vl$, and let $\vltwo$ be a Banach lattice. For $\posop\in\bounded(\vl_0,\vltwo)$, let $\posop^{\tfs{e}}$ denote its extension to an element of $\bounded(\vl,\vltwo)$.
		\begin{enumerate}
			\item\label{part:restriction_is_isomorphism_1} If $t\in\pos{\vltwo}$ is such that $\abs{Tx}\leq \norm{x} t$ for $x\in\vl_0$, then $\abs{T^{\tfs{e}}x}\leq \norm{x} t$ for $x\in\vl$.
			\item\label{part:restriction_is_isomorphism_2} The map $\posop\mapsto\posop^{\tfs e}$ establishes a bipositive linear isomorphism between $\pos{\nob(\vl_0,\vltwo)}-\pos{\nob(\vl_0,\vltwo)}$ and $\pos{\nob(\vl,\vltwo)}-\pos{\nob(\vl,\vltwo)}$;
			\item\label{part:restriction_is_isomorphism_3} When $\vltwo$ is \Dc, the map $\posop\mapsto\posop^{\tfs e}$ is an isometric vector lattice isomorphism between the normed vector lattices $\nob(\vl_0,\vltwo)$ and $\nob(\vl,\vltwo)$.
		\end{enumerate}
	\end{lemma}
	
	Combining \cref{res:norm_to_order_bounded_closed_ideal} and \cref{res:restriction_is_isomorphism} yields the following. It applies, in particular, to the norm dense vector sublattice $\contcts$ of $\contots$.
	
	\begin{theorem}\label{res:isometric_isomorphism_between_banach_lattices_via_restriction} Let $\vl_0$ be a norm dense vector lattice of a normed vector lattice $\vl$, and let $\vltwo$ be an order continuous Banach lattice. When $\pos{\unitball(\vl_0)}$ is upward directed \uppars{equivalently: when $\pos{\unitball(\vl)}$ is upward directed}, the restriction map $\posop\mapsto\posop\!\!\restriction_{\vl_0}$ is an isometric vector lattice isomorphism between the norm closed ideals $\nob(\vl,\vltwo)$ resp.\ $\nob(\vl_0,\vltwo)$ in the Banach lattices $\regular(\vl,\vltwo)$ resp.\ $\regular(\vl_0,\vltwo)$.
	\end{theorem}
	
	We shall now establish a class of vector lattices $\vl$ with the property that $\regular(\contots,\vl)=\nob(\contots,\vl)$ for locally compact Hausdorff spaces $\ts$. For this, we need a few definitions.

	We recall that a vector lattice $\vl$ is said to be \emph{normal} when the order continuous part $\vl^\sim_{\tfs{oc}}$ of its order dual separates the points. Equivalently, when $x\in\vl$ is such that $x^\prime(x)\geq 0$ for every $x^\prime\in\left(\vl^\sim_{\tfs{oc}}\right)^+$, then $x\in\pos{\vl}$. An order continuous Banach lattice is normal. We also recall from \cite[Definition~6.4]{de_jeu_jiang:2022b} that a vector lattice is said to be \emph{quasi-perfect}\footnote{A vector lattice is called \emph{perfect} if the natural vector lattice homomorphism from $\os$ into $(\vl^\sim_{\tfs{oc}})^\sim_{\tfs{oc}}$ is a surjective isomorphism. By a result of Nakano's (see  \cite[Theorem~1.71]{aliprantis_burkinshaw_POSITIVE_OPERATORS_SPRINGER_REPRINT:2006}), this is equivalent to $\vl$ having the following properties:
		\begin{enumerate}
			\item
			$\os$ is normal;
			\item
			if an increasing net $\net{x}$ in $\posos$ is such that $\sup_\lambda x^\prime(x_\lambda)<\infty$ for each $x^\prime\in\pos{(\ocdualos)}$, then this net has a supremum in $\os$.
	\end{enumerate}} when it has the following properties:
	\begin{enumerate}
		\item\label{2_part:quasi_perfect_spaces_1}
		$\os$ is normal;
		\item\label{2_part:quasi_perfect_spaces_2}
		if an increasing net $\net{x}$ in $\posos$ is such that $\sup_\lambda x^\prime(x_\lambda)<\infty$ for each $x^\prime\in\pos{(\odualos)}$, then this net has a supremum in $\os$.
	\end{enumerate}
	It is clear from the second requirement that a quasi-perfect vector lattice is \Dc.
	
	Finally, we recall (see \cite[p.~232]{aliprantis_burkinshaw_POSITIVE_OPERATORS_SPRINGER_REPRINT:2006}) that a Banach lattice is a KB-space when every norm bounded increasing net in the positive cone is norm convergent, and that the norm on a Banach lattice is said to be a Levi norm if every norm bounded increasing net in the positive cone has a supremum. The KB-spaces are precisely the Banach lattices with an order continuous Levi norm. All AL-spaces\footnote{We recall that a Banach lattice is called an \emph{AL-space} when $x\wedge y=0$ in $\bl$ implies that $\norm{x\vee y}=\norm{x}+\norm{y}$. A Banach lattice is an AL-space if and only if it is lattice isometric to an $\Ell_1(\mu)$-space; see \cite[Theorem~4.27]{aliprantis_burkinshaw_POSITIVE_OPERATORS_SPRINGER_REPRINT:2006}.} and all reflexive Banach lattices (in particular: all $\Ell_p(\mu)$-spaces for positive, possibly infinite, real-valued measures $\mu$ and $1<p<\infty$) are KB-spaces; see \cite[p.~232]{aliprantis_burkinshaw_POSITIVE_OPERATORS_SPRINGER_REPRINT:2006}. The Banach lattice $c_0$ has an order continuous norm, but it is not a KB-space.
	
	The following collects facts from \cite[Lemma~6.6, Proposition~6.7, and the proof of Theorem~6.10]{de_jeu_jiang:2022b}.
	
	\begin{proposition}\label{res:examples_of_quasi-perfect_vector_lattices}
		The following vector lattices are quasi-perfect:
		\begin{enumerate}
			\item perfect vector lattices \uppars{in particular: KB-spaces};
			\item normal Banach lattices with Levi norms \uppars{in particular: $\regular(\vl)$ for KB-spaces $\vl$};
			\item the self-adjoint parts of commutative von Neumann algebras.
		\end{enumerate}
	\end{proposition}
	
	The following result was essentially already observed in \cite[proof of Theorem~6.8]{de_jeu_jiang:2022b}. We include the argument for the convenience of the reader.
	
	\begin{proposition}\label{res:regular_is_nob_for_quasi-perfect}
		Let $\ts$ be a non-empty locally compact Hausdorff space, and let $\vl$ be a quasi-perfect vector lattice. Then $\regular(\contots,\vl)=\nob(\contots,\vl)$.
	\end{proposition}
	
	\begin{proof}
		Take $\posop\in\regular(\contots,\vl)$.  For $x^\prime$ in $\pos{(\odualos)}$, consider the positive linear functional $f\mapsto (x^\prime\circ \abs{\posmap})(f)$ from $\contots$ to $\RR$. Its positivity implies that it is continuous. Hence $\{x^\prime(\abs{\posmap}(f)): f\in\contots, \, \zerofunction\leq f\leq \onefunction \}$ is bounded in $\RR$ for each $x^\prime$ in $\pos{(\odualos)}$. Since the set $\{\abs{\posmap}(f): f\in\contots, \, \zerofunction\leq f\leq \onefunction \}$ is upward directed because $\pos{\unitball(\contots)}$ is directed and $\abs{T}$ is positive, the fact that $\os$ is quasi-perfect now implies that it has a supremum $s$ in $\os$. Then $\abs{\posop f}\leq s$ for all $f\in\unitball(\contots)$, showing that $\posop$ is norm to order bounded.
	\end{proof}

	\section{Cones and vector lattices of measures}\label{sec:cones_and_vector_lattices_of_measures}
	
	\noindent We shall now study various cones and vector lattices of measures. In \cref{sec:riesz_representation_theorems}, these will be used to represent cones and vector lattices of regular operators with.
	
	\subsection{Cones of $\pososext$-valued measures}\label{subsec:structure_of_the set_of_positive_measures}
	
	Let $\ms$ be a measurable space, and let  $\os$ be a \sDc\ vector lattice. In this section, we establish basic properties of $\posextmeas$ and $\posmeas$.
	We define a partial ordering on $\posextmeas$ and $\posmeas$ by saying that $\npm_1\leq\npm_2$ when $\npm_1(\mss)\leq\npm_2(\mss)$ in $\osext$ for all $\mss\in\alg$.
	
	For $\npm$, $\npn\in\posextmeas$ and $r\in\posR$, define $\npm+\npn,r\npm:\alg\to\pososext$ by setting
	\begin{align*}
		(\npm+\npn)(\mss)&\coloneqq\npm(\mss)+\npn(\mss)\\
		(r\npm)(\mss)&\coloneqq r\npm(\mss)
	\end{align*}
	for $\mss\in\alg$. Using part~\partref{part:binary_operations_in_extended_space_2} of \cref{res:binary_operations_in_extended_space} (resp.\ part~\partref{part:operations_in_extended_space_7} of \cref{res:operations_in_extended_space}), one sees easily that $\npm+\npn\in\posextmeas$ (resp.\ $r\npm\in\posextmeas$). Clearly, $\npm+\npn$ is finite if and only if $\npm$ and $\npn$ are; for $r>0$, $r\npm$ is finite if and only if $\npm$ is. Thus $\posextmeas$ (resp.\ $\posmeas$) is a subcone of the cone of all set maps from $\alg$ into $\pososext$ (resp.\ $\posos$). As the next result shows, the cone $\posmeas$ is also a lattice when $\os$ is \Dc. In its proof and that of \cref{res:increasing_nets_of_measures_have_a_supremum}, we shall say that a set map $\xi\colon\alg\to\pososext$ is \emph{monotone} when $\xi(\mss_1)\leq\xi(\mss_2)$ if $\mss_1,\mss_2\in\Omega$ are such that $\mss_1\subseteq\mss_2$.
	
	\begin{theorem}\label{res:lattice_cone_of_positive_measures}
		Let $\ms$ be a measurable space, and let $\os$ be a \Dc\ vector lattice.
		\begin{enumerate}
			\item\label{part:lattice_properties_of_measures_1}
			For $\npm$ and $\npn$ in $\posextmeas$, their supremum $\npm\vee\npn$ in $\posextmeas$ exists. We have
			\begin{equation}\label{eq:sup_formula}
				(\npm\vee\npn)(\mss)=\setsup\set{\npm(\msstwo)+\npn(\mss\setminus \msstwo):\msstwo\in\alg \text{ and } \msstwo\subseteq\mss}
			\end{equation}
			in $\osext$ for $\mss\in\alg$. Furthermore, $\npm\vee\npn$ is finite if and only if $\mu$ and $\nu$ are.
			\item \label{part:lattice_properties_of_measures_2}
			For $\npm,\npn\in\posmeas$, their infimum $\npm\wedge\npn$ in $\posextmeas$ exists. It is a finite measure. We have
			\begin{equation}\label{eq:inf_formula}
				(\npm\wedge\npn)(\mss)=\setinf\set{\npm(\msstwo)+\npn(\mss\setminus \msstwo):\msstwo\in\alg \text{ and } \msstwo\subseteq\mss}
			\end{equation}
			in $\osext$ for $\mss\in\alg$.
			\item\label{part:lattice_properties_of_measures_3}
			For $\npm,\npn\in\posmeas$, we have
			\[
			\npm\vee\npn+\npm\wedge\npn=\npm+\npn.
			\]
		\end{enumerate}
	\end{theorem}
	
	The \Dc ness of $\os$ ensures that the suprema occurring in \cref{eq:sup_formula} exist in $\osext$; see part~\partref{part:completeness_properties_of_extended_space_3} of \cref{res:completeness_properties_of_extended_space}.
	
	\begin{proof}
		We prove part~\partref{part:lattice_properties_of_measures_1}. For this, we define a set map  $\xi:\alg\rightarrow\pososext$ by setting
		\begin{equation*}\label{eq:supremum_of_measures}
			\xi(\mss)\coloneqq\setsup\set{\npm(\msstwo)+\npn(\mss\setminus \msstwo):\msstwo\in\alg \text{ and } \msstwo\subseteq\mss}
		\end{equation*}
		for $\mss\in\alg$. We shall verify that $\xi\in\posextmeas$.
		
		It is easy to see that $\xi(\emptyset)=0$ and that $\xi$ is monotone.  
		
		We shall now establish the finite additivity of $\xi$. Suppose that $\mss_1,\mss_2\in\alg$ and that $\mss_1\cap\mss_2=\emptyset$. Then the measurable subsets $\msstwo$ of $\mss_1\cup\mss_2$ are precisely the subsets of the form $\msstwo_1\cup \msstwo_2$, where $\msstwo_1,\msstwo_2\in\alg$ are such that $\msstwo_1\subseteq\mss_1$ and $\msstwo_2\subseteq\mss_2$. Using part~\partref{part:binary_operations_in_extended_space_1} of \cref{res:binary_operations_in_extended_space} in the penultimate step, we therefore see that
		\begin{align*}
			\xi(\mss_1\cup \mss_2)&=\setsup\set{\npm(\msstwo)+\npn((\mss_1\cup \mss_2)\setminus \msstwo): \msstwo\in\alg,\,\msstwo\subseteq \mss_1\cup \mss_2}\\
			&=\setsup\{\,\npm(\msstwo_1\cup \msstwo_2)+\npn((\mss_1\cup \mss_2)\setminus (\msstwo_1\cup \msstwo_2)):\\
			&\quad\quad\quad\msstwo_1,\msstwo_2\in\alg,\,\msstwo_1\subseteq \mss_1,\,\msstwo_2\subseteq\mss_2\,\}\\
			&=\setsup\{\,\npm(\msstwo_1)+\npm(\msstwo_2)+\npn(\mss_1\setminus \msstwo_1)+\npn(\mss_2\setminus \msstwo_2):\\
			&\quad\quad\quad\msstwo_1,\msstwo_2\in\alg,\,\msstwo_1\subseteq \mss_1,\,\msstwo_2\subseteq \msstwo_2\,\}\\
			&=\setsup\set{\npm(\msstwo_1)+\npn(\mss_1 \setminus \msstwo_1): \msstwo_1\in\alg,\, \msstwo_1\subseteq \mss_1}\\
			&\quad\quad+\setsup\set{\npm(\msstwo_2)+\npn(\mss_2 \setminus \msstwo_2): \msstwo_2\in\alg,\, \msstwo_2\subseteq \mss_2}\\
			&=\xi(\mss_1)+\xi(\mss_2).
		\end{align*}
		It follows that $\xi$ is finitely additive.
		
		We shall now show that $\xi$ is $\sigma$-additive. Let $\seq{\mss}$ be a pairwise disjoint sequence in $\alg$, and set $\mss\coloneqq\bigcup_{n=1}^\infty\mss_n$. The facts that $\xi$ is monotone and finitely additive imply that $\xi(\mss)\geq\setsup_{N\geq 1} \sum_{n=1}^N\xi(\mss_n)$. We need to establish the reverse inequality. For every $\msstwo\in\alg$ with $\msstwo\subseteq\mss$, we have, using part~\partref{part:binary_operations_in_extended_space_2} of \cref{res:binary_operations_in_extended_space} in the third step,
		\begin{align*}
			\npm(\msstwo)+\npn(\mss\setminus \msstwo)
			&=\npm\left(\bigcup_{n=1}^\infty(\mss_n\cap \msstwo)\right)+\npn\left(\bigcup_{n=1}^\infty(\mss_n\setminus \msstwo)\right)\\
			&=\setsup_{N\geq 1}\sum_{n=1}^N\npm(\mss_n\cap \msstwo)+\setsup_{N\geq 1}\sum_{n=1}^N\npn(\mss_n\setminus \msstwo)\\
			&=\setsup_{N\geq 1}\sum_{n=1}^N\left(\npm(\mss_n\cap \msstwo)+\npn(\mss_n\setminus \msstwo)\right)\\
			&=\setsup_{N\geq 1}\sum_{n=1}^N\left(\npm(\mss_n\cap \msstwo)+\npn(\mss_n\setminus(\mss_n\cap \msstwo))\right)\\
			&\leq \setsup_{N\geq 1}\sum_{n=1}^N\xi(\mss_n).
		\end{align*}
		Hence $\xi(\mss)\leq \setsup_{N\geq 1}\sum_{n=1}^N\xi(\mss_n)$, as required.
		
		Now that we now that $\xi\in\posextmeas$, the usual argument as in the case $\os=\RR$ shows that $\xi$ is the supremum of $\npm$ and $\npn$ in $\posextmeas$; see \cite[Theorem~36.1]{aliprantis_burkinshaw_PRINCIPLES_OF_REAL_ANALYSIS_THIRD_EDITION:1998}, for example.
		
		We prove part~\partref{part:lattice_properties_of_measures_2}. For this, we define the set map $\rho:\alg\to\pososext$ by setting
		\begin{equation*}\label{eq:infimum_of_measures}
			\rho(\mss)\coloneqq\setinf\set{\npm(\msstwo)+\npn(\mss\setminus \msstwo):\msstwo\in\alg \text{ and } \msstwo\subseteq\mss}.
		\end{equation*}
		for $\mss\in\alg$.
		
		An argument completely analogous to that for $\xi$ above shows that $\rho$ is finitely additive. It remains to show that $\rho$ is $\sigma$-additive.
		
		As a preparation for this, we show that $\rho(\msstwo_n)\downarrow 0$ whenever $\msstwo_1\supseteq \msstwo_2\supseteq\dotsb$ is a decreasing chain in $\alg$ such that $\bigcap_{N=1}^\infty \msstwo_N=\emptyset$. To see this, note that $\rho(\msstwo_N)\leq\npm(\msstwo_N)+\npn(\msstwo_N)$. Since $\npm$ and $\npn$ are now both finite, \cite[Proposition~4.6]{de_jeu_jiang:2022a} shows that $\npm(\msstwo_N)\downarrow 0$ and $\npn(\msstwo_N)\downarrow 0$. Part~\partref{part:binary_operations_in_extended_space_2} of \cref{res:binary_operations_in_extended_space} then yields that $\npm(\msstwo_N)+\npn(\msstwo_N)\downarrow 0$, so that also $\rho(\msstwo_N)\downarrow 0$, as desired.
		
		Let $\seq{\mss}$ be a pairwise disjoint sequence in $\alg$, and set $\mss\coloneqq\bigcup_{n=1}^\infty\mss_n$. Set $\msstwo_1\coloneqq\mss_1$ and $\msstwo_N\coloneqq \mss\setminus\bigcup_{n=1}^{N-1}\mss_n$ for $N\geq 2$. Then $\msstwo_N\downarrow\emptyset$, so $\rho(\msstwo_N)\downarrow 0$. Since $\rho$ is finitely additive, and since $\mss=\msstwo_{N+1}\cup\bigcup_{n=1}^N \mss_n$ is a disjoint union for $N\geq 1$, we see that
		\[
		\rho(\mss)=\rho(\msstwo_{N+1})+\sum_{n=1}^N\rho(\mss_n)
		\]
		for $N\geq 1$, so that
		\[
		\rho(\mss)-\rho(\msstwo_{N+1})=\sum_{n=1}^N\rho(\mss_n)
		\]
		for $N\geq 1$.
		Taking the supremum on each side, we have, since $-\rho(\msstwo_{N+1})\uparrow 0$, that
		\[
		\rho(\mss)=\setsup_{N\geq 1}\sum_{n=1}^N\rho(\mss_n).
		\]
		Hence $\rho$ is $\sigma$-additive.
		
		The usual argument as in the case $\os=\RR$ now shows that $\rho$ is the infimum of $\npm$ and $\npn$ in $\posextmeas$.
		
		The proof for the reals also goes through in general to establish part~\partref{part:lattice_properties_of_measures_3}; see \cite[Theorem~36.1]{aliprantis_burkinshaw_PRINCIPLES_OF_REAL_ANALYSIS_THIRD_EDITION:1998}. 	
	\end{proof}
	
	\begin{remark}\label{rem:counter_example}
		When $\os=\RR$, part~\partref{part:lattice_properties_of_measures_2} (and then also part~\partref{part:lattice_properties_of_measures_3}) of \cref{res:lattice_cone_of_positive_measures} is also valid for infinite $\npm$ and $\npn$. As for the real-valued case of part~\partref{part:lattice_properties_of_measures_1}, this is usually proved using $\varepsilon$-arguments; see \cite[Theorem~36.1]{aliprantis_burkinshaw_PRINCIPLES_OF_REAL_ANALYSIS_THIRD_EDITION:1998}. The above proof shows that the absence of such arguments for general $\os$ can, still for possibly infinite $\mu$ and $\nu$, be circumvented for $\npm\vee\npn$. Part~\partref{part:lattice_properties_of_measures_2}, however, can actually fail for infinite measures in the case of general $\os$, even in the  $\sigma$-finite case. As an example, take a vector lattice $\os$ of dimension at least 2. Choose two non-zero disjoint positive elements $x$ and $y$. We take $\pset=\NN$ and $\alg=2^\NN$. Define $\npm\colon\alg\to\pososext$ by setting $\npm(\mss)\coloneqq \abs{\mss}\cdot x$ and  $\npn(\mss)\coloneqq \abs{\mss}\cdot y$ for $\mss\in\alg$, which is to be read as $\infty$ when the cardinality $\abs{\mss}$ of $\mss$ is infinite. If $\sigma\colon\alg\to\pososext$ is a measure such that $\sigma\leq\npm,\npn$, then $\sigma(\mss)\leq (\abs{\mss}\cdot x)\wedge(\abs{\mss}\cdot y)=0$ when $\abs{\mss}$ is finite. Hence $\sigma=0$. We conclude that $\npm\wedge\npn$ exists in $\posextmeas$ and that it is the zero measure. When $\abs{\mss}$ is finite, $(\npm\wedge\npn)(\mss)$ is still given by the right hand side of \cref{eq:inf_formula}. When $\abs{\mss}$ is infinite, however, and  $\msstwo\subseteq\mss$, at least one of $\msstwo$ and $\mss\setminus \msstwo$ has infinite cardinality. The right hand side of \cref{eq:inf_formula} then equals $\setinf\,\{\infty\}=\infty\neq 0$.
	\end{remark}
	
	We have the following upward order completeness result. The sequential case for $\os=\RR$ is \cite[Theorem~36.2]{aliprantis_burkinshaw_PRINCIPLES_OF_REAL_ANALYSIS_THIRD_EDITION:1998}.
	
	\begin{proposition}\label{res:increasing_nets_of_measures_have_a_supremum}
		Let $\ms$ be a measurable space, and let $\os$ be a \Dc\ vector lattice. Take an increasing net $\net{\npm}$ in $\posextmeas$. Then the set function $\npm\colon\alg\to\pososext$, defined by setting
		\[
		\npm(\mss)\coloneqq\setsup_{\lambda\in\Lambda}\mu_\lambda(\mss)
		\]
		for $\mss\in\alg$, is an $\pososext$-valued measure, and $\npm_\lambda\uparrow\npm$ in $\posextmeas$. If $\npm(\pset)<\infty$, then $\net{\npm}\subseteq\posmeas$ and $\npm_\lambda\uparrow\npm$ in $\posmeas$.
		
		When $\os$ is \sDc, the analogous statements hold for increasing sequences in $\posextmeas$.
	\end{proposition}
	
	\begin{proof}
		We prove both cases simultaneously. Let  $\net{\npm}$ be an increasing net (possibly an increasing sequence) in $\posextmeas$, and define $\npm:\alg\to\pososext$ by setting
		\[
		\npm(\mss)=\setsup\set{\mu_\lambda(\mss) : \lambda\in\Lambda}
		\]
		for $\mss\in\alg$. This supremum exists in $\osext$, as a consequence of either part~\partref{part:completeness_properties_of_extended_space_1} or part~\partref{part:completeness_properties_of_extended_space_3} of \cref{res:completeness_properties_of_extended_space}.
		
		Clearly, $\npm(\emptyset)=0$ and $\npm_\lambda(\mss)\leq\npm(\mss)$ for all $\lambda\in\Lambda$ and $\mss\in\alg$. Since the $\npm_\lambda$ are monotone, so is $\npm$. Furthermore, part~\partref{part:binary_operations_in_extended_space_2} of \cref{res:binary_operations_in_extended_space} implies that the finite additivity of the $\npm_\lambda$ is inherited by $\npm$.
		
		Let $\seq{\mss}$ be a pairwise disjoint sequence in $\alg$, and set $\mss\coloneqq\bigcup_{n=1}^\infty\mss_n$. Then, for all $\lambda\in\Lambda$,
		\[
		\npm_\lambda(\mss)=\setsup_{N\geq 1} \sum_{n=1}^N\npm_\lambda(\mss_n)\leq \setsup_{N\geq 1}\sum_{n=1}^N\npm(\mss_n).
		\]
		Hence $\npm(\mss)\leq\setsup_{N\geq 1}\sum_{n=1}^N\npm(\mss_n)$.
		On the other hand, $\npm$ is monotone and finitely additive, so, for all $N\geq 1$, we have $\npm(\mss)\geq\npm\left(\bigcup_{n=1}^N\mss_n\right)=\sum_{n=1}^N\npm(\mss_n)$. Hence also $\npm(\mss)\geq\setsup_{N\geq 1} \sum_{n=1}^N\npm(\mss_n)$. We conclude that $\npm$ is $\sigma$-additive.
		Now that we know that $\npm\in\posextmeas$, it is evident that $\npm$ is the supremum of the $\mu_\lambda$ in $\posextmeas$. The final statements are clear.
	\end{proof}
	
	\begin{corollary}\label{res:arbitrary_sets_of_measures_have_a_supremum}
		Let $\ms$ be a measurable space, and let $\os$ be a \Dc\ vector lattice. Suppose that $\set{\npm_\lambda : \lambda\in\Lambda}$ is an arbitrary non-empty subset of $\posextmeas$. Then $\setsup_{\lambda\in\Lambda} \npm_\lambda$ exists in $\posextmeas$.
	\end{corollary}
	
	\begin{proof}
		Let $\Phi$ denote the collection of finite subsets of $\Lambda$. For $\phi\in\Phi$, set $\mu_{\phi}\coloneqq\setsup_{\lambda\in\phi}\mu_{\lambda}$; this supremum exists by \cref{res:lattice_cone_of_positive_measures}. The sets $\set{\npm_\phi : \phi\in\Phi}$ and $\set{\npm_\lambda : \lambda\in\Lambda}$ evidently have the same upper bounds in $\posextmeas$. Since $\set{\npm_\phi : \phi\in\Phi}$ is upward directed, \cref{res:increasing_nets_of_measures_have_a_supremum} shows that it has a least upper bound in $\posextmeas$; this is then also the least upper bound of $\set{\npm_\lambda : \lambda\in\Lambda}$.
	\end{proof}
	
	\subsection{Cones of $\pososext$-valued regular Borel measures}\label{sec:cone_of_regular_borel_measures}
	
	Let $\ts$ be a locally compact Hausdorff space, and let $\vl$ be a \Dc\ vector lattice. The results in \cref{subsec:structure_of_the set_of_positive_measures} apply to $\posextmeasts$ and $\posmeasts$.  We shall now show in \cref{res:borel_measures_form_a_lattice_cone,res:regular_measures_form_a_lattice_cone} that various regularity properties are preserved under the cone and lattice operations. The supremum and infimum of two measures in these results
	are those in $\posextmeasts$. Explicit formulas for these are in \cref{eq:sup_formula,eq:inf_formula} in  \cref{res:lattice_cone_of_positive_measures}; these are, however, not needed for the proofs.
	
	Since $\npm\vee\npn\leq\npm+\npn$ and, for finite $\npm$ and $\npn$, also $\npm\wedge\npn\leq\npm$, the following lemma is clear.
	
	\begin{lemma}
		\label{res:borel_measures_form_a_lattice_cone}
		Let $X$ be a locally compact Hausdorff space, and let $\os$ be a \Dc\ vector lattice.
		
		If $\npm,\npn\in\posextmeasts$ are $\pososext$-valued Borel measures and if $r\in\posR$, then $\npm+\npn$, $r\npm$, and $\npm\vee\npn$ are $\pososext$-valued  Borel measures.
		
		If $\npm,\npn\in\posmeasts$ are $\posos$-valued Borel measures, then $\npm\wedge\npn$ is an $\posos$-valued Borel measure.
	\end{lemma}
	
	In the forthcoming proof of \cref{res:regular_measures_form_a_lattice_cone}, we again have to circumvent the use of $\varepsilon$-arguments as in the literature (see \cite[Theorem~38.5]{aliprantis_burkinshaw_PRINCIPLES_OF_REAL_ANALYSIS_THIRD_EDITION:1998}, for example). For this, the following result is convenient.
	
	\begin{lemma}\label{res:regularity_transferred_abstract}
		Let $S$ be a non-empty set, and let $\os$ be a \Dc\ vector lattice. Take $s\in S$ and a non-empty subset $S^\prime$ of $S$.
		\begin{enumerate}
			\item\label{part:regularity_transferred_abstract_1} Suppose that $\npm,\npn\colon S^\prime \cup \set{s}\to\os$ are such that
			\begin{enumerate}
				\item $\npn(s^\prime)\leq \npn (s)$ and $(\npm-\npn)(s^\prime)\leq(\npm-\npn)(s)$ for $s^\prime\in S^\prime$;
				\item $\npm(s)=\setsup_{s^\prime \in S^\prime} \npm(s^\prime)$.
			\end{enumerate}
			Then $\npn(s)=\setsup_{s^\prime \in S^\prime} \npn(s^\prime)$.
			\item\label{part:regularity_transferred_abstract_2}
			Suppose that $\npm,\npn\colon S^\prime \cup \set{s}\to\os$ are such that
			\begin{enumerate}
				\item $\npn(s^\prime)\geq \npn (s)$ and $(\npm-\npn)(s^\prime)\geq(\npm-\npn)(s)$ for $s^\prime\in S^\prime$;
				\item  $\npm(s)=\inf_{s^\prime \in S^\prime} \npm(s^\prime)$.
			\end{enumerate}
			Then  $\npn(s)=\setinf_{s^\prime \in S^\prime} \npn(s^\prime)$.
		\end{enumerate}
	\end{lemma}
	
	\begin{proof}
		We prove part~\partref{part:regularity_transferred_abstract_1}. As $\npn(s^\prime)\leq\npn(s)$ for $s^\prime\in S^\prime$, it is clear that $\setsup_{s^\prime \in S^\prime} \npn(s^\prime)$ exists and that $\setsup_{s^\prime \in S^\prime} \npn(s^\prime)\leq\npn(s)$.
		For the reverse inequality, we argue as follows:
		\begin{align*}
			\npn(s)&=\npm(s)-(\npm-\npn)(s)\\
			&=\setsup_{s^\prime\in S^\prime}\,\big(\npm(s^\prime)-(\npm-\npn)(s)\big)\\
			&\leq \setsup_{s^\prime\in S^\prime}\big(\npm(s^\prime)-(\npm-\npn)(s^\prime)\big)\\
			&=\setsup_{s^\prime\in S^\prime}\big(\npn(s^\prime)\big).
		\end{align*}
		Part~\partref{part:regularity_transferred_abstract_2} follows from an application of part~\partref{part:regularity_transferred_abstract_1} to $-\npm$ and $-\npn$.
	\end{proof}
	
	\begin{corollary}\label{res:regularity_transferred_concrete}
		Let $\ts$ be a locally compact Hausdorff space, and suppose that $\npn,\npm\in\posextmeasts$ are such that $\npn\leq\npm$. Take $\mss\in\borel$, and suppose that $\npm(\mss)$ is finite. If $\npm$ is inner \uppars{resp.\ outer} regular at $\mss$, then $\npn$ is inner  \uppars{resp.\ outer} regular at $\mss$.
	\end{corollary}
	
	\begin{proof}
		When $\npm$ is inner regular at $\mss$, we apply \cref{res:regularity_transferred_abstract} with $S=\borel$, where we take $S^\prime$ to be the compact subsets of $\mss$ and $\set{\mss}$ for $s$. Then part~\partref{part:regularity_transferred_abstract_1} of \cref{res:regularity_transferred_abstract} shows that $\npm$ is inner regular at $\npn$ because $\npm$ is. When $\npm$ is outer regular at $\mss$, we take $S^\prime$ to be the open subsets of $\ts$ that contain $\mss$ and have finite measure, and $s=\set{\mss}$.  Then part~\partref{part:regularity_transferred_abstract_2} of \cref{res:regularity_transferred_abstract} shows that $\npm$ is outer regular at $\npn$ because $\npm$ is.
	\end{proof}

	\begin{proposition}\label{res:regular_measures_form_a_lattice_cone}
		Let $X$ be a locally compact Hausdorff space, and let $\os$ be a \Dc\ vector lattice. Take $\mss\in\borel$.
		
		If $\npm,\npn\in\posextmeasts$ are inner \uppars{resp.\ outer} regular  at $\mss$ and if $r\in\posR$, then $\npm+\npn$, $r\npm$, and $\npm\vee\npn$ are inner \uppars{resp.\ outer} regular at $\mss$.
		
		If $\npm,\npn\in\posmeasts$ are inner \uppars{resp.\ outer} regular at $\mss$, then $\npm\wedge\npn$ is inner \uppars{resp.\ outer}  regular at $\mss$.
	\end{proposition}
	
	\begin{proof}
		We start with $\npm+\npn$.
		
		Suppose that $\npm$ and $\npn$ are inner regular at $\mss$. Then we have, using part~\partref{part:binary_operations_in_extended_space_1} of \cref{res:binary_operations_in_extended_space} in the second step,
		\begin{align*}
			(\npm+&\npn)(\mss)\!
			=\!\setsup\!\set{\npm(K_1)\!:\! K_1\! \text{ is compact and }\! K_1\subseteq \mss}\\
			&\quad\quad\quad\quad+\setsup\set{\npn(K_2)\!:\! K_2 \!\text{ is compact and }\! K_2\subseteq \mss}\\
			&\!=\!\setsup\!\set{\npm(K_1)+\npn(K_2)\!:\! K_i \text{ is compact and }\! K_i\subseteq \mss\!\text{ for }\!\ i=1,2}\\
			&\!\leq\!\setsup\!\set{\npm(K_1\cup K_2)+\!\npn(K_1\cup K_2)\!:\! K_i\! \text{ is compact and }\! K_i\!\subseteq\!\mss\!\text{ for }\!i=1,2}\\
			&\!=\setsup\!\set{\npm(K)+\npn(K)\!:\! K \!\text{ is compact and }\! K\subseteq \mss}\\
			&\!=\setsup\!\set{(\npm+\npn)(K)\!:\! K\! \text{ is compact and }\! K\subseteq \mss}\\
			&\!\leq (\npm+\npn)(\mss).
		\end{align*}
		Hence $\npm+\npn$ is inner regular at $\mss$.
		
		Suppose that $\npm$ and $\npn$ are outer regular at $\mss$. Then we have, again using part~\partref{part:binary_operations_in_extended_space_1} of \cref{res:binary_operations_in_extended_space} in the second step,
		\begin{align*}
			&(\npm+\npn)(\mss)\\
			&=\setinf\set{\!\npm(V_1)\!: V_1 \text{ is open and }\mss\subseteq V_1\!}+\setinf\set{\!\npn(V_2)\!: V_2 \text{ is open and } \mss\subseteq V_2\!}\\
			&=\setinf\set{\npm(V_1)+\npn(V_2): V_i \text{ is open and } \mss\subseteq V_i\text{ for }i=1,2}\\
			&\geq\setinf\set{\npm(V_1\cap V_2)+\npn(V_1\cap V_2): V_i \text{ is open and } \mss\subseteq V_i\text{ for }i=1,2}\\
			&=\setinf\set{\npm(V)+\npn(V): V \text{ is open and } \mss\subseteq V}\\
			&=\setinf\set{(\npm+\npn)(V): V \text{ is open and } \mss\subseteq V}\\
			&\geq (\npm+\npn)(\mss).
		\end{align*}
		Hence $\npm+\npn$ is outer regular at $\mss$.
		
		The statement for $r\npm$ is immediate from part~\partref{part:operations_in_extended_space_7} of \cref{res:operations_in_extended_space}.
		
		We turn to $\npm\vee\npn$.
		
		Suppose that $\npm$ and $\npn$ are inner regular at $\mss$.
		
		If $(\npm\vee\npn)(\mss)=\largest$, then, since $\npm\vee\npn\leq\npm+\npn$, at least one of $\npm(\mss)$ and $\npn(\mss)$ equals $\largest$. Suppose that $\npm(\mss)=\largest$. Then
		\[
		\setsup\set{\npm(K): K \text{ is compact and } K\subseteq \mss}=\largest.
		\]
		Since $\npm\vee\npn\geq\npm$, we then also have that
		\[
		\setsup\set{(\npm\vee\npn)(K): K \text{ is compact and } K\subseteq \mss }=\largest.
		\]
		Arguing similarly when $\npn(\mss)=\largest$, we see that $\npm\vee\npn$ is inner regular at $\ms$.
		
		If $(\npm\vee\npn)(\mss)<\largest$, then $\npm(\mss), \npn(\mss)<\largest$, so $(\npm+\npn)(\mss)<\largest$. As we already know that $\npm+\npn$ is inner regular at $\mss$, and since $\npm\vee\npn\leq\npm+\npn$, \cref{res:regularity_transferred_concrete} shows that $\npm\vee\npn$ is inner regular at $\mss$.
		
		Suppose that $\npm$ and $\npn$ are outer regular at $\mss$.
		
		If $(\npm\vee\npn)(\mss)=\largest$, then, since $(\npm\vee\npn)(\mss)\leq\npm(\mss)+\npn(\mss)$, at least one of $\npm(\mss)$ and $\npn(\mss)$ equals $\largest$. Suppose that $\npm(\mss)=\largest$.
		Then $\npm(V)=\largest$ for all open subsets $V$ such that $\mss\subseteq V$. Since $(\npm\vee\npn)(V)\geq\npm(V)=\largest$ for such $V$, we see that
		\[
		\setinf\set{(\npm\vee\npn)(V): V \text{ is open and }\mss\subseteq V}=\setinf\,\{\largest\}=\largest=(\npm\vee\npn)(\mss).
		\]
		Arguing similarly when $\npn(\mss)=\largest$, we see that $\npm\vee\npn$ is outer regular at $\mss$.
		
		When $(\npm\vee\npn)(\mss)<\infty$, the facts that $\npm\vee\npn\leq\npm+\npn$ and that $(\npm+\npn)(\mss)<\infty$, the outer regularity of $\npm+\npn$ at $\mss$, and \cref{res:regularity_transferred_concrete} imply that $\npm\vee\npn$ is outer regular at $\mss$.

		For $\npm,\npn\in\posrBmeas$, it is immediate from the fact that $\npm\wedge\npn\leq\npm$, the inner (resp.\ outer) regularity of $\npm$ at $\mss$, and \cref{res:regularity_transferred_concrete} that $\npm\wedge\npn$ is inner (resp.\ outer) regular at $\mss$.
	\end{proof}
	
	The following is clear from the \cref{res:borel_measures_form_a_lattice_cone}, \cref{res:regular_measures_form_a_lattice_cone}, and \cref{res:regularity_transferred_concrete}.
	
	\begin{theorem}\label{res:lattice_properties_of_measures_on_topological_spaces}
		Let $X$ be a locally compact Hausdorff space, and let $\os$ be a \Dc\ vector lattice.
		\begin{enumerate}
			\item\label{part:lattice_properties_of_measures_on_topological_spaces_1} $\posextrBmeas$ is a subcone of $\posextmeasts$. For $\npm,\npn\in\posextrBmeas$, their supremum in $\posextrBmeas$ exists and is equal to their supremum in $\posextmeasts$;
			\item\label{part:lattice_properties_of_measures_on_topological_spaces_2} $\posrBmeas$ is a subcone and sublattice of $\posmeasts$;
			\item\label{part:lattice_properties_of_measures_on_topological_spaces_3} Suppose that $\npn\in\posmeasts$ and $\npm\in\posrBmeas$ are such that $\npn\leq\npm$. Then $\npn\in\posrBmeas$.
		\end{enumerate}
	\end{theorem}

	\begin{remark}\label{rem:almost_ideal_real_case}\quad
		Part~\partref{part:lattice_properties_of_measures_on_topological_spaces_3} \cref{res:lattice_properties_of_measures_on_topological_spaces} will be strengthened in \cref{res:regular_borel_measures_form_ideal}.	Even when $\os=\RR$ it appears to be not so widely known. Another, considerably more circumstantial, way to infer the result in this case is as follows. If $0\leq\npn\leq\npm$, then $\npn\ll\npm$, so the Radon-Nikodym theorem furnishes a $\mu$-integrable non-negative function $f$ such that $\npn(\mss)=\int_\mss f\di{\npm}$ for $\mss\in\borel$. By \cite[Exercise~7.2.8]{folland_REAL_ANALYSIS_SECOND_EDITION:1999}, $\npn$ is an $\pos{\RR}$-valued regular Borel measure.
	\end{remark}
	
	\subsection{Vector lattices of $\os$-valued measures}
	
	Let $\os$ be a \sDc\ vector lattice, and let $\ms$ be a measurable space. We define $\smeas$, the \emph{$\os$-valued measures}, to be the linear subspace of the vector space of set maps from $\alg$ into $\os$ that is spanned by $\posmeas$. The remarks preceding \cref{res:lattice_cone_of_positive_measures} imply that the elements of $\smeas$ are of the form $\npn_1-\npn_2$ for $\npn_1,\npn_2\in\posmeas$. If $\seq{\mss}$ is a pairwise disjoint sequence in $\alg$ and $\npm\in\smeas$, then $\npm \left(\bigcup_{n=1}^\infty\mss_n\right)=\tfs{o}\text{-}\!\lim_{N\to\infty} \sum_{n=1}^N \npm(\mss_n)$.

	We supply $\smeas$ with the pointwise ordering.
	
	\begin{lemma}\label{res:positive_cone_as_expected}
		Let $\ms$ be a measurable space, and let $\vl$ be a \Dc\ vector lattice. Then $\pos{\smeas}=\posmeas$.
	\end{lemma}
	
	\begin{proof}	
		If $\npm,\npn\in\posmeas$ are such that $\npm-\npn\geq 0$, and if $\seq{\mss}$ is a pairwise disjoint sequence in $\alg$, then
		\begin{align*}
			(\npm-\npn)\left(\bigcup_{n=1}^\infty \mss_n\right)&=\tfs{o}\text{-}\!\!\!\lim_{N\to\infty} \sum_{n=1}^N (\npm-\npn)(\mss_n)\\
			&=\setsup_{N\geq 1}\sum_{n=1}^N (\npm-\npn)(\mss_n).
		\end{align*}
		Hence $\npm-\npn$ is $\sigma$-additive.
	\end{proof}
	
	A partially ordered vector space with a generating cone is a vector lattice if and only if every pair of positive elements has a supremum. Translating back  and forth, we thus have the following consequence of \cref{res:lattice_cone_of_positive_measures},  \cref{res:increasing_nets_of_measures_have_a_supremum}, and \cref{res:positive_cone_as_expected}.
	
	\begin{theorem}\label{res:vector_lattice_of_E-valued_measures}
		Let $\ms$ be a measurable space, and let $\os$ be a \Dc\ vector lattice. Then $\pos{\smeas}=\posmeas$, and $\smeas$ is a \Dc\ vector lattice.
		For $\npm$ and $\npn$ in $\smeas$, we have
		\begin{align}
			(\npm\vee\npn)(\mss)&=\setsup\set{\npm(\msstwo)+\npn(\mss\setminus \msstwo):\msstwo\in\alg \text{ and } \msstwo\subseteq\mss}\label{eq:sup_formula_E-valued}\\
			\intertext{and}
			(\npm\wedge\npn)(\mss)&=\setinf\set{\npm(\msstwo)+\npn(\mss\setminus \msstwo):\msstwo\in\alg \text{ and } \msstwo\subseteq\mss}\label{eq:inf_formula_E-valued}		
		\end{align}
		for $\mss\in\alg$. If $\net{\npm}$ is an increasing net in $\smeas$ that is bounded from above, then
		\[
		\left(\setsup_{\lambda\in\Lambda}\npm_\lambda\right)(\mss)=  \setsup_{\lambda\in\Lambda}\left(\npm_\lambda(\mss)\right)
		\]
		for $\mss\in\alg$; similarly for the infimum of a decreasing net that is bounded from below.
	\end{theorem}
	
	\begin{remark}\label{rem:hahn}
		When $\os$ in \cref{res:vector_lattice_of_E-valued_measures} is a vector lattice, one has
		\begin{equation}\label{eq:triangle_inequality_for_measure}
			\abs{\npm(\mss)}=\abs{\npm^+(\mss)-\npm^-(\mss)}\leq \npm^+(\mss)+\npm^-(\mss)=\abs{\npm}(\mss)
		\end{equation}
		for $\mss\in\alg$. Hence, just as in the real case (see \cite[Theorem~36.9]{aliprantis_burkinshaw_PRINCIPLES_OF_REAL_ANALYSIS_THIRD_EDITION:1998}),
		\[
		\abs{\npm}(\mss)\geq\setsup\lrset{\sum_{i=1}^n\abs{\npm(\mss_i)}: \mss_1,\dotsc,\mss_n\subseteq\mss\text{ are pairwise disjoint}}.
		\]
		for $\npm\in\smeas$ and $\mss\in\alg$. In the real case, it is a consequence of the Hahn decomposition that equality holds. For general $\os$, a Hahn decomposition need not exist. As an example, let $\pset=\set{p}$ be a one-point set and take $\alg=\set{\emptyset, \set{p}}$. Take $\os=\RR^2$, and set $\npm(\emptyset)\coloneqq (0,0)$ and $\npm(\set{p})\coloneqq (1,-1)$. Then $\pos{\npm}(\set{p})=(1,0)$ and $\npm^-(\set{p})=(0,1)$. Clearly, there does not exist a disjoint partition $\pset=\mss^+\cup \mss^-$ such that $\npm^\pm(\set{p})=\npm(\set{p}\cap \mss^\pm)$.
	\end{remark}

	\subsection{Vector lattices of $\os$-valued regular Borel measures}\label{subsec:E-valued_sign_measure_on_locally_compact_Hausdorff_space}
	
	Let $\ts$ be a locally compact Hausdorff space, and let $\vl$ be a \Dc\ vector lattice. We define $\rBsmeas$, the \emph{$\os$-valued regular Borel measures \uppars{on $\ts$}}, to be the linear subspace of the vector space of set maps from $\alg$ to $\os$ that is spanned by $\posrBmeas$. As a consequence of part~\partref{part:lattice_properties_of_measures_on_topological_spaces_2} of \cref{res:lattice_properties_of_measures_on_topological_spaces}, the elements of $\rBsmeas$ are of the form $\npm-\npn$ for $\npm,\npn\in\posrBmeas$.  We supply $\rBsmeas$ with the pointwise ordering.

	\begin{lemma}\label{res:positive_cone_as_expected_Borel} Let $\ts$ be a locally compact Hausdorff space, and let $\vl$ be a \Dc\ vector lattice.
		Then $\pos{\rBsmeas}=\posrBmeas$.
	\end{lemma}
	
	\begin{proof}
		Suppose that $\npm-\npn\geq 0$ for some $\npm,\npn\in\posrBmeas$. We know from \cref{res:positive_cone_as_expected} that $\npm-\npn\in\posmeasts$. As $0\leq\npm-\npn\leq\npm$, part~\partref{part:lattice_properties_of_measures_on_topological_spaces_3} of \cref{res:lattice_properties_of_measures_on_topological_spaces} implies that $\npm-\npn\in\posrBmeas$.
	\end{proof}
	
	\begin{theorem}\label{res:regular_borel_measures_form_ideal}
		Let $\ts$ be a locally compact Hausdorff space, and let $\vl$ be a \Dc\ vector lattice. Then $\pos{\rBsmeas}=\posrBmeas$, and $\rBsmeas$ is an ideal in the \Dc\ complete vector lattice $\smeasts$.
	\end{theorem}
	
	\begin{proof}
		We start by showing that $\npm^\pm\in\posrBmeas$ for $\npm\in\rBsmeas$. Writing $\npm=\npn_1-\npn_2$ for some $\npn_1,\npn_2\in\posrBmeas$, we have $\npm^+=(\npn_1\vee\npn_2)-\npn_2$. Since $0\leq \npm^+\leq \npn_1\vee\npn_2$, and since we know $\npn_1\vee\npn_2$ to be in $\posrBmeas$ by part~\partref{part:lattice_properties_of_measures_on_topological_spaces_2} of \cref{res:lattice_properties_of_measures_on_topological_spaces},
		part~\partref{part:lattice_properties_of_measures_on_topological_spaces_3} of \cref{res:lattice_properties_of_measures_on_topological_spaces} shows that $\npm^+\in\posrBmeas$. Likewise, $\npm^-\in\posrBmeas$.

		Suppose that $\npn\in\smeasts$ and $\npm\in\rBsmeas$ are such that $\abs{\npn}\leq\abs{\npm}$. Then $0\leq\npn^\pm\leq\npm^++\npm^-$. Part~\partref{part:lattice_properties_of_measures_on_topological_spaces_2} of \cref{res:lattice_properties_of_measures_on_topological_spaces} yields that $\npm^++\npm^-\in\posrBmeas$. Then part~\partref{part:lattice_properties_of_measures_on_topological_spaces_3} of \cref{res:lattice_properties_of_measures_on_topological_spaces} shows that $\npn^\pm\in\posrBmeas$, implying that $\npn\in\rBsmeas$.
	\end{proof}
	
	\subsection{Banach lattices of $\bl$-valued measures}\label{sec:banach_lattice_of_measures}
	
	Let $\ms$ be a measurable space, and let $\bl$ be a \Dc\ normed vector lattice. For $\npm\in\smeas$, set
	\begin{equation}\label{eq:norm_definition}	
		\norm{\npm}\coloneqq\norm{\,\abs{\npm}(\pset)\,}.
	\end{equation}
	Observing that, for $\mss\in\alg$,
	\begin{equation}\label{eq:uniform_convergence}
		\norm{\npm(\mss)}\leq\norm{\,\abs{\npm}(\mss)\,}\leq\norm{\,\abs{\npm}(\pset)\,}=\norm{\npm}
	\end{equation}
	as a consequence of \cref{eq:triangle_inequality_for_measure}, it is easily verified that this defines a vector lattice norm on $\smeas$.
	
	The following result is convenient to understand how various properties transfer between $\bl$ and $\smeas$.
	
	\begin{proposition} \label{res:contains_copy} 
		Let $\ms$ be a measurable space, and let $\bl$ be a \Dc\ complete Banach lattice. Suppose that $\pset\neq\emptyset$. Take $x\in\vl$. For $e\in\bl$, define $\npm_e\in\smeas$  by setting
		\begin{equation*}
			\npm_e(\mss)\coloneqq
			\begin{cases}
				e&\text{ if }x\in\mss;\\
				0&\text{ if }x\notin\mss.
			\end{cases}
		\end{equation*}
		for $\mss\in\alg$. Define $\psi\colon\bl\to\smeas$ by setting $\psi(e)=\npm_e$ for $e\in\bl$. Then $\psi$ is an isometric embedding of $\bl$ as a band in $\smeas$.
	\end{proposition}
	
	\begin{remark}
		Although we shall not need this, we note the following expression for the band projection $P\colon\smeas\to\psi(\bl)$ in \cref{res:contains_copy}. For $\npm\in\posmeas$ and $\mss\in\alg$, it is not difficult to verify that
		\begin{equation*}
			[P(\npm)](\mss)=
			\begin{cases}
				\inf \set{\npm(\msstwo):\msstwo\in\alg,\,x\in\msstwo}&\text{ if }x\in\mss;\\
				0&\text{ if }x\notin\mss.
			\end{cases}	
		\end{equation*}
	\end{remark}

	\begin{theorem}\label{res:banach_lattice_of_signed_measures_general}
		Let $\ms$ be a measurable space, and let $\bl$ be a \Dc\ Banach lattice. Then $\smeas$ is a \Dc\ Banach lattice in the norm from \cref{eq:norm_definition}. Furthermore:
		\begin{enumerate}
			\item\label{part:banach_lattice_of_signed_measures_general_1} if $\bl$ is an AL-space, then so is $\smeas$; when $\pset\neq\emptyset$, the converse holds;
			\item\label{part:banach_lattice_of_signed_measures_general_2} if $\bl$ is order continuous, then so is $\smeas$; when $\pset\neq\emptyset$, the converse holds;
			\item\label{part:banach_lattice_of_signed_measures_general_3} if $\bl$ has a Levi norm, then so has $\smeas$; when $\pset\neq\emptyset$, the converse holds;
			\item\label{part:banach_lattice_of_signed_measures_general_4} if $\bl$ is a KB-space, then so is $\smeas$; when $\pset\neq\emptyset$, the converse holds.
		\end{enumerate}
	\end{theorem}

	\begin{proof}
		According to \cite[Theorem~16.2]{zaanen_INTRODUCTION_TO_OPERATOR_THEORY_IN_RIESZ_SPACES:1997}, the completeness of $\smeas$ is equivalent to it having the Riesz-Fischer property. Suppose, therefore, that $\{\npn_k\}_{k=1}^\infty\subseteq\posmeas$ is such that $\sum_{k=1}^\infty\norm{\npn_k}<\infty$. Setting $\npm_n\coloneqq\sum_{k=1}^n \npn_k$ for $n\geq 1$, we have to show that the increasing sequence $\seq{\npm}$ in $\posmeas$ has a supremum in $\posmeas$. For $\mss\in\alg$, we have $\sum_{k=1}^\infty\norm{\npn_k(\mss)}\leq \sum_{k=1}^\infty\norm{\npn_k(\pset)}=\sum_{k=1}^\infty\norm{\npn_k}<\infty$. Hence $\sum_{k=1}^\infty \npn_k(\mss)$ converges in norm, and then $\npm_n(\mss)\uparrow\sum_{k=1}^\infty \npn_k(\mss)$. According to \cref{res:increasing_nets_of_measures_have_a_supremum}, setting $\npm(\mss)\coloneqq\sum_{k=1}^\infty \npn_k(\mss)$ for $\mss\in\alg$ defines an element $\npm$ of $\posmeas$, and $\npm_n\uparrow\npm$. Thus $\smeas$ is a Banach lattice. Its Dedekind completeness follows from \cref{res:vector_lattice_of_E-valued_measures}.
		
		We prove part~\partref{part:banach_lattice_of_signed_measures_general_1}.
		
		Suppose that $\bl$ is an AL-space. Then we have, for $\npm,\npn\in\posmeas$,
		\[
		\norm{\npm}+\norm{\npn}=\norm{\npm(\pset)}+\norm{\npn(\pset)}=\norm{\npm(\pset)+\npn(\pset)}=\norm{\npm+\npn},
		\]
		so that $\smeas$ is an AL-space. Conversely, when $\smeas$ is an AL-space and $\pset\neq\emptyset$, \cref{res:contains_copy} implies that $\bl$ is an AL-space.
		
		We prove part~\partref{part:banach_lattice_of_signed_measures_general_2}.
		Suppose that $\bl$ is order continuous. Let $\net{\npm}\subseteq\posmeas$ be a net such that $\npm_\lambda\downarrow 0$. According to \cref{res:vector_lattice_of_E-valued_measures} we have $\npm_\lambda(\pset)\downarrow 0$. Hence $\norm{\npm_\lambda}=\norm{\npm_\lambda(\pset)}\downarrow 0$, so that $\smeas$ is order continuous.  When $\smeas$ is order continuous and $\pset\neq\emptyset$, \cref{res:contains_copy} implies that $\bl$ is order continuous.
		
		We prove part~\partref{part:banach_lattice_of_signed_measures_general_3}.
		
		Suppose that $\bl$ has a Levi norm. Let $\net{\npm}$ be a norm bounded increasing net in $\posmeas$. Then $\{\npm_\lambda(\ts)\}_{\lambda\in\Lambda}$ is a norm bounded increasing net in $\posos$. Hence $\npm_\lambda(\ts)\uparrow e$ for some $e\in\posos$. We know from \cref{res:increasing_nets_of_measures_have_a_supremum} that setting $\npm(\mss)\coloneqq\sup_\lambda \npm_\lambda(\mss)$ for $\mss$ defines an $\pososext$-valued measure on $\alg$. As $\npm(X)=e$, $\npm$ is finite, and then $\npm_{\lambda}\uparrow\npm$ in $\smeas$ by  \cref{res:increasing_nets_of_measures_have_a_supremum}. Hence $\smeas$ has a Levi norm.
		
		Since the property of having a Levi norm inherits from a Banach lattice to a band in it, \cref{res:contains_copy} implies that $\bl$ has a Levi norm when $\smeas$ does and $\pset\neq\emptyset$.
		
		Part~\partref{part:banach_lattice_of_signed_measures_general_4} can be proved similarly to part~\partref{part:banach_lattice_of_signed_measures_general_3}. As the KB-spaces are precisely the Banach lattices with an order continuous Levi norm, it also follows from the parts \partref{part:banach_lattice_of_signed_measures_general_2} and \partref{part:banach_lattice_of_signed_measures_general_3} combined.
	\end{proof}

	\subsection{Banach lattices of $\bl$-valued regular Borel measures}\label{sec:banach_lattice_of_regular_borel_measures}
	
	We conclude our study of cones and vector lattices of measures with the following.
	
	\begin{theorem}\label{res:banach_lattice_of_signed_measures_topological_spaces}
		Let $\ts$ be a locally compact Hausdorff space, and let $\bl$ be an order continuous Banach lattice. Then $\smeasts$ is an order continuous Banach lattice, and $\rBsmeas$ is a band in it.  Furthermore:
		\begin{enumerate}
			\item\label{part:banach_lattice_of_signed_measures_topological_space_1} if $\bl$ is an AL-space, then so is $\rBsmeas$; when $\pset\neq\emptyset$, the converse holds;
			\item\label{part:banach_lattice_of_signed_measures_topological_space_2} if $\bl$ has a Levi norm, then so has $\rBsmeas$; when $\pset\neq\emptyset$, the converse holds;
			\item\label{part:banach_lattice_of_signed_measures_topological_space_3} if $\bl$ is a KB-space, then so is $\rBsmeas$; when $\pset\neq\emptyset$, the converse holds.
		\end{enumerate}
	\end{theorem}
	
	\begin{proof}
		The order continuity of $\rBsmeas$ is a special case of \cref{res:banach_lattice_of_signed_measures_general}. We show that $\rBsmeas$ is a band. Since $\rBsmeas$ is an ideal in $\smeas$ by \cref{res:regular_borel_measures_form_ideal}, and since the (projection) bands in an order continuous Banach lattice are precisely its norm closed ideals (see \cite[Corollary~2.4.4]{meyer-nieberg_BANACH_LATTICES:1991}), it is sufficient to show that $\rBsmeas$ is norm closed.  Take a sequence $\seq{\npm}\subseteq\rBsmeas$ and suppose that $\npm_n\to\npm$. Then also $\npm_n^\pm\to\npm^\pm$. Hence it is sufficient to show that $\npm\in\posrBmeas$ when $\npm_n\to\npm$ and $\seq{\npm}\subseteq\posrBmeas$.
		
		Take an open subset $V$ of $\ts$. We shall show that $\npm$ is inner regular at $V$.  Let $\varepsilon>0$. Take an $n_0$ such that $\norm{\npm-\npm_{n_0}}<\varepsilon/3$, so that $\norm{\npm(\mss)-\npm_{n_0}(\mss)}<\varepsilon/3$ for all $\mss\in\alg$ by \cref{eq:uniform_convergence}. Let $\mathscr K$ denote the collection of compact subsets of $V$, ordered by inclusion. As $\npm_{n_0}$ is inner regular at $V$, we have $\{\npm_{n_0}(K)\}_{K\in\mathscr K}\uparrow \npm_{n_0}(V)$, and then $\npm_{n_0}(K)\to\npm_{n_0}(V)$ by the order continuity of the norm of $\bl$. Hence there exists a $K_0\in\mathscr K$ such that $\norm{\npm_{n_0}(V)-\npm_{n_0}(K)}<\varepsilon/3$ for all $K\in\mathscr K$ with $K\supseteq K_0$. We then have $\norm{\npm(V)-\npm(K)}<\varepsilon$ for all $K\in\mathscr K$ with $K\supseteq K_0$. Thus $\npm(K)\to\npm(V)$, and then $\{\npm(K)\}_{K\in\mathscr K}\uparrow \npm(V)$ as the net is increasing.
		
		The proof that $\npm$ is outer regular at all elements of $\alg$ is similar.
		
		The forward implications in the parts~\partref{part:banach_lattice_of_signed_measures_topological_space_1}, \partref{part:banach_lattice_of_signed_measures_topological_space_2} and \partref{part:banach_lattice_of_signed_measures_topological_space_3} follow from \cref{res:banach_lattice_of_signed_measures_general} as the pertinent properties of $\smeasts$ are inherited by the band $\rBsmeas$ in it. For the converse implications when $\pset\neq\emptyset$, we note that, in the notation of \cref{res:contains_copy}, the band $\psi(\bl)$ in $\smeasts$ is also a band in the ideal $\rBsmeas$ of $\smeasts$. Hence it inherits the pertinent properties from $\rBsmeas$.
	\end{proof}
	
	\begin{remark}\label{rem:band}\quad
		\begin{enumerate}
			\item Let $P_{\tfs{rB}}\colon\smeasts\to\rBsmeas$ be the order projection corresponding to the projection band $\rBsmeas$.  Take $\npm\in\posmeasts$ and suppose that $\npm\in\posrBmeas$ is such that $\npn\leq\npm$. Then $0\leq \npn\leq P_{\tfs{rB}}\npm$. Hence $P_{\tfs{rB}}\npm$ is the largest $\posos$-valued regular Borel measure on $\ts$ that is dominated by $\npm$.
			\item For $\os=\RR$, the fact that $\spaceofmeasuresletter_{\tfs{rB}}(\ts,\borel,\RR)$ is a band in $\spaceofmeasuresletter(\ts,\borel,\RR)$ is a special case of \cite[437F(b)]{fremlin_MEASURE_THEORY_VOLUME_4:2006} and of \cite[Theorem~2.8(i)]{van_amstel_van_der_walt:2024}.
		\end{enumerate}
	\end{remark}
	
	\section{Riesz representation theorems for vector lattices and Banach lattices of regular operators}\label{sec:riesz_representation_theorems}
	
	\noindent In this section, we shall establish a number of isomorphisms between vector lattices (resp.\ cones) of $\os$-valued measures (resp.\ $\pososext$-valued regular Borel measures) and vector lattices (resp.\ cones) of regular (resp.\ positive) operators. When the domain and codomain of these isomorphisms are Banach lattices, the isomorphisms are isometric.
	
	The general set-up for vector lattices of $\os$-valued measures is as follows. For cones of $\pososext$-valued regular Borel measures it is similar.\footnote{It is a little simpler as there is no need to decompose an $\pososext$-valued regular measure into two positive ones. There are no integrability issues as the only functions to be considered are those in $\contcts$.}.
	
	Let $\ms$ be a measurable space, and let $\vl$ be a \Dc\ vector lattice. Take $\mu\in\smeas$, and choose a decomposition $\npm=\npn_1-\npn_2$ with $\npn_1,\npn_2\in\posmeas$. We let $\boundedmeasfun$ denote the bounded measurable real-valued functions on $\pset$,  and define $\opint{\npm}\colon\boundedmeasfun\to\vl$ by setting
	\begin{equation*}
		\opint{\npm}(f)\coloneqq\orderintegral{\pset}{f}{\npn_1}-\orderintegral{\pset}{f}{\npn_2}
	\end{equation*}
	for $f\in\boundedmeasfun$. For $\xi_1,\xi_2\in\posmeas$, it is a consequence of the definition of the integral that $\orderintegral{\pset}{f}{\xi_1}+\orderintegral{\pset}{f}{\xi_2}=\orderintegral{\pset}{f}{(\xi_1+\xi_2)}$ for $f\in\boundedmeasfun$. It follows from this that the definition of $\opint{\npm}$ is independent of the choice of the decomposition $\npm=\npn_1-\npn_2$. Using that, also as consequence of the definition of the integral, $\orderintegral{\pset}{f}{(\alpha\npm)}= \alpha\orderintegral{\pset}{f}{\npm}$ for $f\in\boundedmeasfun$, $\alpha\geq 0$, and $\npm\in\posmeas$, one verifies easily that the map $\npm\mapsto\opint{\npm}$ defines a positive operator $\opint{}\colon\smeas\to\regular(\boundedmeasfun, E)$. On restricting the $\npm$ to be taken from a vector sublattice $N(\pset,\alg,\vl)$ of $\smeas$, and on restricting the ensuing $\opint{\npm}$ to a vector sublattice $\vltwo$ of $\boundedmeasfun$ one obtains a positive operator
	\[
	\opint{}\colon N(\pset,\alg,\vl) \to\regular(\vltwo, E).
	\]
	We have used the same symbol $\opint{}$ here, suppressing the dependence on $N(\pset,\alg,\vl)$ and $\vltwo$, and shall do likewise in the sequel to avoid a further burdening of the notation.
	
	For the decomposition $\npm=\npn^+-\npn^-$, the triangle inequality in \cref{eq:triangle_inequality} implies that
	\begin{equation*}
		\abs{\opint{\npm} (f)}\leq \norm{f}\cdot\abs{\npm}(\pset)
	\end{equation*}
	for $f\in\boundedmeasfun$ and $\npm\in\smeas$.
	Therefore, when the vector sublattice $\vltwo$ of $\boundedmeasfun$ is supplied with the uniform norm, the map $\npm\mapsto\opint{\npm}$ defines, in fact, a positive operator
	\[
	\opint{}\colon N(\pset,\alg,\vl)\to\nob(\vltwo,\vl).
	\]
	
	In the contexts to be considered below, $\opint{}\colon N(\pset,\alg,\vl)\to\nob(\vltwo,\vl)$ is a vector lattice homomorphism. In this case, when $\vltwo$ is supplied with the uniform norm, and when $\vl$ is a \Dc\ Banach lattice, we have, for $\npm\in N(\pset,\alg,\vl)$,
	\begin{equation}\label{eq:formula_for_regular_norm}
		\begin{split}	
			\rnorm{\opint{\npm}}&=\norm{\,\abs{\opint{\npm}}\,}\\
			&=\norm{\opint{\abs{\npm}}}\\
			&=\sup\set{\norm{\opint{\abs{\npm}}(f)}:f\in\vltwo,\,\zerofunction\leq f\leq\onefunction}.
		\end{split}
	\end{equation}
	Hence $\rnorm{\opint{\npm}}\leq \norm{\,\abs{\npm}(\pset)\,}$, and
	\begin{equation}\label{eq:norm_of_integral_operator}
		\rnorm{\opint{\npm}}=\norm{\,\abs{\npm}(\pset)\,}
	\end{equation}
	when $\onefunction\in\vltwo$. When $\vltwo$ is $\contcts$ or $\contots$ for a locally compact Hausdorff space, the results from \cite{de_jeu_jiang:2022b} that we use yield that \eqref{eq:norm_of_integral_operator} still holds, even though $\onefunction$ is not in $\contcts$ or $\contots$ when $\ts$ is not compact.
	
	Along these lines, we shall now establish a number of representation theorems for vector lattices and Banach lattices of norm to order bounded operators. In these results, it is understood that, when applicable, the norm on a vector lattice of regular operators is the regular norm, and that the norm on a vector lattice of $E$-valued measures is as in \cref{eq:norm_definition}, i.e., $\norm{\npm}=\norm{\,\abs{\npm}(\pset)\,}$.
	
	We start with cases where the domains are the spaces $\contcts$ or $\contots$ for non-empty locally compact Hausdorff spaces $\ts$. Our basis here consists of the following two representation theorems. They are \cite[Theorem~5.4]{de_jeu_jiang:2022b} and a slight specialisation of \cite[Theorem~4.2]{de_jeu_jiang:2022b}, respectively, formulated in the terminology of the present paper.
	
	\begin{theorem}\label{res:riesz_representation_theorem_for_contcts_finite_normal_case}
		Let $\ts$ be a non-empty locally compact Hausdorff space, and let $\os$ be a \Dc\ normal vector lattice. Take $\posop\in\pos{\nob(\contcts,\vl)}$. Then there exists a unique $\npm\in\posextrBmeas$ such that
		\[
		\posmap(f)=\ointm{f}
		\]
		for $f\in\contcts$. The measure $\npm$ is finite. If $V$ is a non-empty open subset of $\ts$, then
		\begin{equation*}
			\npm(V)=\setsup\set{\posmap(f) : f\in\contcts,\, \zerofunction\leq f\leq\onefunction,\, \supp f\subseteq V}
		\end{equation*}
		in $\os$. If $K$ is a compact subset of $\ts$, then
		\begin{equation*}
			\npm(K)=\setinf\set{\posmap(f) : f\in\contcts, \zerofunction\leq f\leq\onefunction,\  f\!\!\restriction_{K}=\onefunction\!\!\restriction_{K}}
		\end{equation*}
		in $\os$.
	\end{theorem}
	
	\begin{theorem}\label{res:riesz_representation_theorem_for_contcts_normed_case}
		Let $\ts$ be a non-empty locally compact Hausdorff space, and let $\os$ be an order continuous Banach lattice. Take $\posop\in\pos{\regular(\contcts,\vl)}$. Then there exists a unique $\npm\in\posextrBmeas$ such that
		\[
		\posmap(f)=\ointm{f}
		\]
		for $f\in\contcts$. If $V$ is a non-empty open subset of $\ts$, then
		\begin{equation*}
			\npm(V)=\setsup\set{\posmap(f) : f\in\contots,\, \zerofunction\leq f\leq\onefunction,\, \supp f\subseteq V}
		\end{equation*}
		in $\osext$. If $K$ is a compact subset of $\ts$, then
		\begin{equation*}
			\npm(K)=\setinf\set{\posmap(f) : f\in\contots, \zerofunction\leq f\leq\onefunction,\  f\!\!\restriction_{K}=\onefunction\!\!\restriction_{K}}
		\end{equation*}
		in $\osext$. Hence $\npm$ is a finite measure if and only if $\posop\in\pos{\nob(\contcts,\vl)}$.
	\end{theorem}

	Our first result applies to all order continuous Banach lattices.
	
	\begin{theorem}\label{res:Dedekind_complete_normal}
		Let $\ts$ be a non-empty locally compact Hausdorff space, and let $\vl$ be a \Dc\ normal vector lattice.
		Then the maps
		\begin{enumerate}
			\item\label{part:Dedekind_complete_normal_1} $\opint{}\colon\rBsmeas\to\nob(\contcts,\vl)$ and
			\item\label{part:Dedekind_complete_normal_2} $\opint{}\colon\rBsmeas\to\nob(\contots,\vl)$
		\end{enumerate}
		are surjective vector lattice isomorphisms.
		
		When $\ts$ is compact, the spaces $\nob(\contots,\vl)$ and $\nob(\contcts,\vl$) both coincide with $\regular(\contts,\vl)$.
	\end{theorem}
	
	\begin{proof}
		We prove part~\partref{part:Dedekind_complete_normal_1}. We already know from \cref{res:riesz_representation_theorem_for_contcts_finite_normal_case}, that $\opint{}$ maps $\pos{\rBsmeas}$ onto $\pos{\nob(\contcts,\vl)}$, implying that $\opint{}$ is surjective. If $\opint{\npm}=0$, then $\opint{\npm^+}=\opint{\npm^-}$. By the uniqueness part of \cref{res:riesz_representation_theorem_for_contcts_normed_case}, we have $\npm^+=\npm^-$, so $\npm=0$. Hence $\opint{}$ is injective. Since $\opint{}\colon\pos{\rBsmeas}\to\pos{\nob(\contcts,\vl)}$ is surjective, $\opint{}$ is a bipositive surjection.

		We prove part~\partref{part:Dedekind_complete_normal_2}. Take $\posop\in\pos{\nob(\contots,\vl)}$. Again by \cref{res:riesz_representation_theorem_for_contcts_finite_normal_case}, there exists a $\npm\in\posrBmeas$ such that $\posop(f)=\opint{\npm}(f)$ for $f\in\contcts$. Take $x^\prime\in\pos{(\ocdual{\vl})}$, and consider the positive linear functionals $f\mapsto (x^\prime\circ \posop)(f)$ and $f\mapsto (x^\prime\circ\opint{\npm})(f)$ on $\contots$. They are continuous because $\contots$ is a Banach lattice; since they agree on $\contcts$, they are equal. As $\vl$ is normal, this implies that $\posop(f)=\opint{\npm}(f)$ for $f\in\contots$. Hence $\posop=\opint{\npm}$. The remainder of the proof is as in the proof of part~\partref{part:Dedekind_complete_normal_1}.
		
		The final statement is clear.
	\end{proof}
	
	Consequently, we have the following extension theorem, which is not a consequence of \cref{res:restriction_is_isomorphism} as $\vl$ need not be a Banach lattice.
	
	\begin{corollary}
		Let $\ts$ be a non-empty locally compact Hausdorff space, and let $\vl$ be a \Dc\ normal vector lattice. Then the restriction map induces a vector lattice isomorphism between $\nob(\contots,\vl)$  and $\nob(\contcts,\vl)$.
	\end{corollary}
	
	When $\vl$ is not just \Dc\ and normal but even quasi-perfect, we know from \cref{res:regular_is_nob_for_quasi-perfect} that $\regular(\contots,\vl)=\nob(\contots,\vl)$. This is the case, for example, when $\vl$ is a KB-space or equal to $\regular(E)$ for a KB-space $\vl$; see \cref{res:examples_of_quasi-perfect_vector_lattices}.  \cref{res:Dedekind_complete_normal} then yields the first part of the next result. The proof of its second part is valid for quasi-perfect order continuous Banach lattices, but a moment's thought shows that these are simply the KB-spaces.
	
	\begin{corollary}\label{res:quasi-perfect}
		Let $\ts$ be a non-empty locally compact Hausdorff space.
		\begin{enumerate}
			\item\label{part:quasi-perfect_1} When $\vl$ is a quasi-perfect vector lattice, the map $\opint{}\colon\rBsmeas\to\regular(\contots,\vl)$ is a surjective vector lattice isomorphism.
			\item\label{part:quasi-perfect_2} When $\vl$ is a KB-space, the map $\opint{}\colon\rBsmeas\to\regular(\contots,\vl)$ is an isometric surjective vector lattice isomorphism between Banach lattices. Furthermore, these Banach lattices are AL-spaces if and only if $\vl$ is.
		\end{enumerate}
	\end{corollary}
	
	\begin{proof}
		For part~\partref{part:quasi-perfect_2}, we recall from \cref{eq:formula_for_regular_norm} that
		\[
		\rnorm{\opint{\npm}}=\sup\set{\norm{\opint{\abs{\npm}}(f)}:f\in\contots,\,\zerofunction\leq f\leq\onefunction}\leq\norm{\abs{\npm}(X)}.
		\]
		On the other hand, we have from \cref{res:riesz_representation_theorem_for_contcts_finite_normal_case} that
		\[
		\abs{\npm}(\ts)=\sup\set{\opint{\abs{\npm}}(f):f\in\contcts,\,\zerofunction\leq f\leq\onefunction}.
		\]
		The order continuity of $\vl$ then implies that
		\[
		\norm{\abs{\npm}(\ts)}=\sup\set{\norm{\opint{\abs{\npm}}(f)}:f\in\contcts,\,\zerofunction\leq f\leq\onefunction}\leq \rnorm{\opint{\npm}}.
		\]
		Hence $I$ is isometric. The final statement is from \cref{res:banach_lattice_of_signed_measures_topological_spaces}.
		
	\end{proof}
	
	When $\vl$ is an order continuous Banach lattice, \cref{res:norm_to_order_bounded_closed_ideal} implies that $\nob(\contcts,\vl)$ and $\nob(\contots,\vl)$ are Banach lattices. They are isometrically isomorphic by part~\partref{part:restriction_is_isomorphism_3} of \cref{res:restriction_is_isomorphism}. This is also seen from our next result.
	
	\begin{theorem}\label{res:order_continuous_banach_lattice_2}
		Let $\ts$ be a non-empty locally compact Hausdorff space, and let $\vl$ be an order continuous Banach lattice.
		Then the maps
		\begin{enumerate}
			\item\label{part:order_continuous_banach_lattice_2_a} $\opint{}\colon\rBsmeas\to\nob(\contcts,\vl)$ and
			\item\label{part:order_continuous_banach_lattice_2_b} $\opint{}\colon\rBsmeas\to\nob(\contots,\vl)$
		\end{enumerate} are isometric surjective vector lattice isomorphisms between Banach lattices.  Furthermore, these Banach lattices are AL-spaces, have a Levi norm, or are KB-spaces if and only if $\vl$ has the pertinent property.
		
		When $\ts$ is compact, the spaces $\nob(\contots,\vl)$ and $\nob(\contcts,\vl$) both coincide with $\regular(\contts,\vl)$.
	\end{theorem}
	
	\begin{proof}
		The fact that $\opint{}$ is a surjective vector lattice isomorphism in both cases is a specialisation of \cref{res:Dedekind_complete_normal}.
		The proof that it is an isometry is in both cases as in the proof of part~\partref{part:quasi-perfect_2} of \cref{res:quasi-perfect}.
		The penultimate statement is from \cref{res:banach_lattice_of_signed_measures_topological_spaces}, and the final one is clear.
	\end{proof}

	We conclude our treatment of norm to order bounded operators with a representation theorem that is independent of the earlier material in this paper and of \cite{de_jeu_jiang:2022b}.
	
	\begin{theorem}\label{res:bounded_measurable_functions_as_domain}
		Let $\ms$ be a measurable space where $\pset\neq\emptyset$, and let $\vl$ be a \Dc\ vector lattice. Then $\opint{}\colon\smeas\to\soc(\boundedmeasfun,\vl)$ is a surjective vector lattice isomorphism. When $\vl$ is a \Dc\ Banach lattice, $\opint{}$ is an isometry between Banach lattices. In this case, these Banach lattices are AL-spaces, are order continuous, have a Levi norm, or are KB-spaces if and only if $\vl$ has the pertinent property.
	\end{theorem}
	
	\begin{proof}
		If $\seq{f}\subset\boundedmeasfun$ and $f\in\boundedmeasfun$, then $f_n\overset{\sigma\tfs{oc}}{\rightarrow}f$ if and only if $\set{f_n:f\geq 1}$ is order bounded and $f_n\to f$ pointwise. An appeal to the dominated convergence theorem (see \cite[Theorem~6.13]{de_jeu_jiang:2022a}) therefore yields that, indeed, $\opint{\npm}\in\soc(\boundedmeasfun,\vl)$ for $\npm\in\smeas$.
		
		To see that $\opint{}\colon\smeas\to\soc(\boundedmeasfun,\vl)$ is surjective, take a $\posop$ in $\soc(\boundedmeasfun,\vl)$. Then $\posop^\pm\in\soc(\boundedmeasfun,\vl)$. For $\mss\in\alg$, set $\npm^\pm(\mss)=\posop^\pm(\chi_\mss)$. The order continuity of $\posmap^\pm$ implies that $\npm^\pm\in\posmeas$. It follows easily from the definition of the integral and the $\sigma$-order continuity of $\posop^\pm$ that $\posmap^\pm=\opint{\npm^\pm}$, so that $\posop=\opint{\npm^+-\npm^-}$. Since $\opint{\npm}(\chi_\mss)=\npm(\mss)$ for $\npm\in\smeas$ and $\mss\in\alg$ by construction, $\opint{}$ is bipositive. Hence $\opint{}$ is a vector lattice isomorphism.
		
		When $\vl$ is a \Dc\ Banach lattice, we have $\rnorm{\opint{\npm}}=\norm{\abs{\npm}(\pset)}$ by \cref{eq:norm_of_integral_operator} as $\onefunction\in\boundedmeasfun$.
		
		The final statement is from \cref{res:banach_lattice_of_signed_measures_general}.
	\end{proof}

	\begin{remark}
		The proofs that the operators $\opint{}$ in the above representation theorems in this section are bipositive surjections do not use that the domains are vector lattices. The fact that the codomains are vector lattices is sufficient for this, and it \emph{implies} that the domains are vector lattices. Whereas this is hardly an efficient way to establish the latter, it is of some interest to use the isomorphism in  \cref{res:bounded_measurable_functions_as_domain} to determine what the lattice operations in $\smeas$ are. Observing that $\boundedmeasfun$ has the principal projection property as it is \sDc, we have, for $\npm,\npn\in\smeas$ and $\mss\in\alg$, using \cite[Theorem~1.50]{aliprantis_burkinshaw_POSITIVE_OPERATORS_SPRINGER_REPRINT:2006} in the third step, that
		\begin{align*}
			(\npm\vee\npn&)(\mss)=\opint{\npm\vee\npn}(\chi_\mss)\\
			&=(\opint{\npm}\vee\opint{\npn})(\chi_\mss)\\
			&=\!\setsup\!\set{\!\opint{\npm}(f_1)+\opint{\npn}(f_2)\!:\!f_1,f_2\!\in\!\boundedmeasfun, f_1+f_2=\chi_\mss, f_1\wedge f_2=0\!}\\
			&=\!\setsup\!\set{\!\opint{\npm}(\chi_\msstwo)+\opint{\npn}(\chi_\mss-\chi_\msstwo):\msstwo\in\alg\text{ and }\msstwo\subseteq\mss\!}\\
			&=\!\setsup\!\set{\!\npm(\msstwo)+\npn(\mss\setminus\msstwo):\msstwo\in\alg\text{ and }\msstwo\subseteq\mss\!}.
		\end{align*}
		We have thus retrieved the formula for the supremum of two $\os$-valued measures in \cref{eq:sup_formula_E-valued}. The formula for the infimum in \cref{eq:inf_formula_E-valued} can similarly be seen as a Riesz-Kantorovich formula for operators.

		When $\vl$ is a \Dc\ Banach lattice, the vector lattice isomorphism in \cref{res:bounded_measurable_functions_as_domain} allows us to \emph{introduce} a vector lattice norm on $\smeas$ by setting $\norm{\npm}\coloneqq\rnorm{\opint{\npm}}$ for $\npm\in\smeas$, i.e., by setting $\norm{\npm}\coloneqq\norm{\,\abs{\npm}(\pset)\,}$. We know that $\soc(\boundedmeasfun,\vl)$ is a band in the Banach lattice $\regular(\boundedmeasfun,\vl)$, so it is norm closed. We have thus retrieved the completeness statement in \cref{res:banach_lattice_of_signed_measures_general}.
		
		When $\ts$ is a locally compact Hausdorff space, one can infer the formulas for the supremum and infimum for $\rBsmeas$ by showing that it is a vector sublattice of $\smeasts$. For this, one proceeds as in \cref{sec:cone_of_regular_borel_measures}. The first step is to show that $\posrBmeas$ is closed under addition. This is easy; see the proof of \cref{res:regular_measures_form_a_lattice_cone}). The second step is to show that $\npn\in\posrBmeas$ when $\npn\in\posmeasts$ and $\npn\leq\npm$ for some $\npm\in\posrBmeas$, which is an immediate consequence of \cref{res:regularity_transferred_abstract}. After that, one proves as in \cref{res:regular_borel_measures_form_ideal} that $\rBsmeas$ is even an ideal in $\smeasts$.
		
	\end{remark}

	Finally, we establish a representation theorem for vector lattices of regular operators that need not be norm to order bounded. Clearly, infinite measures will then be needed.
	
	\begin{theorem}\label{res:order_continuous_banach_lattice_1}
		Let $\ts$ be a non-empty locally compact Hausdorff space, and let $\vl$ be an order continuous Banach lattice. Then:
		\begin{enumerate}
			\item\label{part:order_continuous_banach_lattice_1}the cone map $\opint{}\colon\posextrBmeas\to\pos{\regular(\contcts,\vl)}$ is a surjective order isomorphism;
			\item\label{part:order_continuous_banach_lattice_2} for $\posop\in\regular(\contcts,\vl)$, there exist $\npm_1,\npm_2\in\posextrBmeas$ such that $\posop(f)=\opint{\npm_1}(f)-\opint{\npm_2}(f)$ for $f\in\contcts$.
		\end{enumerate}
	\end{theorem}
	
	\begin{proof}
		We prove part~\partref{part:order_continuous_banach_lattice_1}. It is immediate from \cref{res:riesz_representation_theorem_for_contcts_normed_case} that $\opint{}$ is an injective surjection. It is obviously increasing. Suppose that $\opint{\npm}\leq\opint{\npn}$. By \cref{res:riesz_representation_theorem_for_contcts_normed_case} we have, for an open subset $V$ of $\ts$,
		\[
		\npm(V)=\setsup\set{\opint{\npm}(f): f\in\contcts,\, \zerofunction\leq f\leq\onefunction,\, \supp f\subseteq V},
		\]
		and similarly for $\npn$. Hence $\npm(V)\leq\npn(V)$ for all open subsets $V$ of $\ts$. The outer regularity of $\npm$ and $\npn$ then implies that  $\npm\leq\npn$.
		
		Part part~\partref{part:order_continuous_banach_lattice_2} is immediate from part~\partref{part:order_continuous_banach_lattice_1}.
	\end{proof}
	
	\begin{remark}
		For $\vl=\RR$, part~\partref{part:order_continuous_banach_lattice_1} of \cref{res:order_continuous_banach_lattice_2} gives the familiar isometric isomorphism between the real-valued regular Borel measures on $\ts$ and the norm dual $\contcts^\ast$ of $\contcts$ (see \cite[Theorem~38.7]{aliprantis_burkinshaw_PRINCIPLES_OF_REAL_ANALYSIS_THIRD_EDITION:1998}). Part~\partref{part:order_continuous_banach_lattice_2} of \cref{res:order_continuous_banach_lattice_1} determines its order dual $\odual{\contcts}$.
	\end{remark}

	\subsection*{Acknowledgements} The authors thank Jan Harm van der Walt for providing the references in \cref{rem:band}, and the anonymous referee for their detailed reading of the manuscript. The results in this paper were obtained, in part, during visits of the first author to Sichuan University and of the second author to Leiden University. The generous support by the Erasmus+ ICM programme that made these possible is gratefully acknowledged.  
	
	\bibliographystyle{plain}
	\urlstyle{same}
	
	\bibliography{general_bibliography}
	
\end{document}